%% file: CompLiap.tex
\documentclass[preprint,12pt]{elsarticle}
\usepackage{amssymb,amsmath,amsthm,graphicx}
\usepackage{hyperref}

\newtheorem{lemma}{Lemma}[section]
\newtheorem{proposition}{Proposition}[section]
\newtheorem{corollary}{Corollary}[section]
\newtheorem{theorem}{Theorem}[section]
\theoremstyle{definition}
\newtheorem{definition}{Definition}[section]
\newtheorem{assumption}{Assumption}[section]

\newcommand{\In}{\subseteq}

\newcommand{\IN}{\mathbb{N}}
\newcommand{\IZ}{\mathbb{Z}}
\newcommand{\IQ}{\mathbb{Q}}
\newcommand{\IR}{\mathbb{R}}
\newcommand{\ld}{\mathrm{ld}}
\newcommand{\ivl}[1]{{\boldsymbol #1}}

\journal{Mathematics and Computers in Simulation}

\begin{document}

\begin{frontmatter}

\title{Computational Complexity of Iterated Maps on the Interval}
\author{Christoph Spandl}
\address{Computer Science Department,
Universit\"{a}t der Bundeswehr M\"{u}nchen,
D-85577 Neubiberg, Germany}
\ead{christoph.spandl@unibw.de}

\begin{abstract}
The correct computation of orbits of discrete dynamical systems on the
interval is considered. Therefore, an arbitrary-precision floating-point
approach based on automatic error analysis is chosen and a general algorithm
is presented. The correctness of the algorithm is shown and the
computational complexity is analyzed. There are two main results.
First, the computational complexity measure considered here is related
to the Lyapunov exponent of the dynamical system under consideration.
Second, the presented algorithm is optimal with regard to that
complexity measure.
\end{abstract}
\begin{keyword}
Discrete dynamical systems \sep Lyapunov exponent \sep arbitrary-precision
floating-point arithmetic

\MSC 37M05 \sep 65P20
\end{keyword}

\end{frontmatter}

\section{Introduction}

Consider a discrete dynamical system $(D,f)$ on some compact interval $D\In\IR$,
called the phase space, given by a function $f:D\to D$, a recursion relation
$x_{n+1}=f(x_n)$ and an initial value $x_0\in D$. The sequence $(x_n)_n$ of
iterates is called the orbit of the dynamical system in phase space corresponding
to the initial value $x_0$. If such a dynamical system is implemented, that is
a computer program is written for calculating a finite initial segment of
the orbit for given $x_0$, care has to be taken in choosing the appropriate
data structure for representing real numbers. Traditionally, IEEE 754 {\tt double}
floating-point numbers \cite{ieee08, go91} are used. However,
if the dynamical system shows chaotic behavior, a problem arises. The finite
and constant length of the significand of a {\tt double} variable causes rounding
errors which are magnified after each iteration step. Already after a few
iterations, the error is so big that the computed values are actually
useless. For example in \cite{lo93, Mue01} this phenomenon is examined for the
dynamical system $(D,f)$ with $D=[0,1]$, $f(x)=3.75\cdot x\cdot(1-x)$ and
the initial value $x_0=0.5$. To put things right, a rigorous method
for computations with real numbers has to be used. There already exist
some rigorous numerical methods in the field of dynamical systems and chaos
\cite{hy87, gh90, sy91, nr93, rn94, mr95}.

In the next section, a rigorous method based on arbitrary-precision floating
point arithmetic is presented and used to investigate the iteration of a
generalization of the above mentioned function. Correctness of the results are
obtained by using a method called running error analysis. The method and the
numerics are compared to interval arithmetic. In Section \ref{main-sec},
the algorithm is generalized to arbitrary functions $f$. The aim of the present
paper is to give bounds on some kind of space complexity of the algorithm.
To be more precise, the behavior of the the length of the significand
in arbitrary-precision arithmetic is analyzed in the task of iterates of
discrete dynamical systems. The minimal length of the significand needed for
floating-point numbers such that any computed point of an initial segment of
the orbit has a specified and guaranteed accuracy is examined. This minimal
length  will be related to the length of the initial segment of the orbit.
To cope with this task, a precise mathematical framework for floating-point
computations is applied. This framework should be suited to computability concepts
over the reals. Finally, a complexity measure for describing the computational
effort on arbitrary-precision floating-point numbers is introduced. Roughly
speaking, it is the ratio of the length of the significand to the number of
iterations in the limit of number of iterations to infinity. The first main
result shows that this complexity measure is related to the Lyapunov exponent.
The second main result proves that the presented algorithm 
to compute the orbit up to any given accuracy is optimal with respect
to that complexity measure. As a consequence, these results give some advice
for economically designing reliable algorithms simulating one-dimensional
discrete dynamical systems.

\section{Dynamic behavior of the logistic equation and rounding error}

In this section, the discrete dynamical system $(D,f_\mu)$ with $D=[0,1]$ and
$f_\mu:D\to D$, $f_\mu(x):=\mu x(1-x)$ for some control parameter $\mu\in(0,4]$
is investigated. In the literature, the recursion relation $x_{n+1}=f_\mu(x_n)$ is
called the {\em logistic equation} \cite{ce80}. When implementing the logistic
equation on a real computer and demanding to obtain true values for the
orbit $(x_n)_n$, some rigorous method is needed. Since for some values of
$\mu$ the dynamics is highly chaotic, inaccuracies are magnified exponentially in
time \cite{ce06, hsd04}. Therefore, it is clear that using floating-point numbers
with a predefined, fixed precision makes sense only if the maximum iteration
time $N$ also is a predefined, fixed number. If the algorithm should work
for any $N$, a more elaborate approach is needed. First one can
work with arbitrarily high precision floating-point numbers, the precision
dynamically set and the error control implemented in the algorithm separately.
A software package for doing this task is for example MPFR \cite{Fousse:2007:MMP}.
Second, there are methods with automatic
error control, for example interval arithmetic \cite{mo66,ah83}, the Feasible Real
RAM model \cite{BH98} or significance arithmetic \cite{mrt75}. Implementations
are for example MPFI \cite{ReRo02}, the iRRAM \cite{Mue01} and Mathematica
\cite{ss05} respectively.

All these methods have the same theoretical background. They are all practical
instances of the model of Computable Analysis \cite{wh00,pr89,ko91} used in
computer science.
While the Feasible Real RAM model directly implements the theory of Computable
Analysis, the other mentioned methods all have their background in scientific
computing. Looking closer at the various validated methods in use, they all have
in common implementing some kind of intervals for representing real numbers
numerically. Therefore, the starting point here is looking at interval arithmetic
for computing orbits $(x_n)_n$. For any time step $n$, let the phase point
$x_n$ together with its computational error be represented by two floating-point
numbers $x^l_n$ and $x^u_n$ ($x^u_n\geq x^l_n$) with given length $m_n$ of the
significand, called the precision, forming an interval $[x^l_n,x^u_n]$. The
interval is an enclosure of the real value $x_n$, that is $x_n\in[x^l_n,x^u_n]$
for all $n$. The interval length $d_n:=x^u_n-x^l_n$ gives a measure of the
uncertainty about $x_n$ and is therefore a kind of error. Interval arithmetic
often models quantities which are not known exactly. But here, the true orbit
$(x_n)_n$ can be, in principle, calculated to any given accuracy. Thus, in
the present setting, the true object of interest is not an interval, but an
approximation $\hat{x}_n$ of $x_n$ together with an absolute error $e_n$. The
interval is only used for mathematical convenience. To transform the interval
to a floating point value $\hat{x}_n$ of precision $m_n$,  just do
\begin{equation}
\label{val_def}
\hat{x}_n:=rd\left(\frac{x^l_n + x^u_n}{2},m_n\right)
\end{equation}
where $rd(x,m)$ performs a rounding to some floating-point number of
precision $m$ nearest to $x$.
Note that rounding to nearest is not unique if $x$ is equidistant
from two floating-point numbers.
The absolute error $e_n:=|\hat{x}_n - x_n|$ of $\hat{x}_n$ can
be estimated via the interval length $d_n$ by
\begin{equation}
\label{err_def}
e_n\leq\frac{1}{2}d_n + r_n
\end{equation}
where $r_n$ is an error caused by the rounding operation $rd(.,.)$ in
Equation (\ref{val_def}). An upper bound on $r_n$ will be discussed later,
for now it suffices to say that in general it is small compared to $d_n$.

The aim now is to calculate, for given initial value $x=x_0$, $N\in\IN$
and $p\in\IZ$ the orbit up to time $N$
with relative error at most $10^{-p}$. That is, for
$(\hat{x}_n)_{0\leq n\leq N}$ it should hold
\begin{equation}
\label{prec_req}
e_n=|\hat{x}_n - x_n|\leq 10^{-p}x_n \leq 10^{-p}.
\end{equation}
Why using here and in the following the relative error and not the absolute
error is discussed in some detail at the end of Subsection \ref{comp-and-corr}.
The minimal $m$, fulfilling the precision requirement (\ref{prec_req}) on
the relative error of $x_n$, which depends on $x$, $N$ and $p$, is
denoted by $m_{min}(x,N,p)$. Now, a central
quantity of this work is introduced, which is some complexity measure.
Consider the growth rate of $m_{min}(x,N,p)$,
\begin{equation*}
\sigma(x,p)=\limsup_{N\to\infty}\frac{m_{min}(x,N,p)}{N}.
\end{equation*}
The {\em loss of significance rate} $\sigma(x)$, which may depend on
the initial value $x$ is given by
\begin{equation*}
\sigma(x)=\lim_{p\to\infty}\sigma(x,p).
\end{equation*}
This quantity describes the limiting amount of significant precision
being lost at each iteration step in the limit of infinite output
precision. Significant means here the part of
the digits being correct. A general treatment of this complexity measure
is given in the next section. Roughly speaking,
$\lceil\sigma(x_0,p) N + p\cdot\ld(10)\rceil$ is the precision
for any floating-point number needed in an algorithm doing the iteration
starting with $x_0$ and calculating to $x_N$, if the output should be
precise to at least $p$ decimal places. Here,  $\ld(.)$ is the logarithm
to base $2$.

\subsection{Dynamic behavior of the logistic equation}
Before analyzing the different numerical behavior, it is worth having an
analytical look at the dynamical behavior of the system. Despite
the fact that these results are well known \cite{hsd04, de89}, they are
reviewed here for the sake of self containment.

First have a look at the fixed points of the logistic equation and their stability.
In the range $D=[0,1]$, the equation possesses exactly one
fixed point $x^o=0$ if $\mu\in(0,1]$ and exactly two fixed points $x^o=0$
and $x^{(\mu)}=1-\frac{1}{\mu}$ if $\mu\in(1,4]$. Looking at the derivatives
$f'_\mu(x^o)=\mu$ and $f'_\mu(x^{(\mu)})=2-\mu$ gives the stability of the
fixed points. Since $|f'_\mu(x^o)|<1$ for $\mu\in(0,1)$ and $|f'_\mu(x^o)|>1$
for $\mu\in(1,4]$, $x^o$ is a stable fixed point, an attractor for
$\mu\in(0,1)$ and an unstable fixed point, a repeller for $\mu\in(1,4]$.
If $\mu=1$, the only fixed point $x^o$ is hyperbolic, that is $|f'_1(x^o)|=1$.
At $\mu=1$, a bifurcation occurs. If
$\mu\in(1,3)$, $x^o$ becomes unstable and the newly occurring fixed point
$x^{(\mu)}$ is stable. At $\mu=3$ a second bifurcation occurs and for $\mu>3$
both fixed points are unstable.

Second, examine the basin of attraction of the stable fixed point.
If $\mu\in(0,1)$, the contraction mapping principle directly gives
$\lim_{n\to\infty}f^n_\mu(x)=x^o$ for all $x\in[0,1]$. If $\mu=1$,
observe that $f_1(x)<x$ holds for all $x\in(0,1)$. Hence, any sequence
$(f^n_1(x))_n$, $x\in(0,1)$, is strictly decreasing and bounded from
below. So, also $\lim_{n\to\infty}f^n_1(x)=x^o$ holds for all $x\in[0,1]$.
Finally, in the case $\mu\in(1,3)$, $\lim_{n\to\infty}f^n_\mu(x)=x^{(\mu)}$
holds for all $x\in[0,1]$. For a proof, the interested reader is referred
to the literature:\cite{de89}, Proposition 5.3 in Section 1.5.

Finally, for $\mu>3$ the system goes into a region showing periodic
behavior with period doubling bifurcations. Finally, for some
$\mu<4$, chaotic behavior is reached.

This analysis shows that in the parameter range $\mu\in(0,3)$,
the orbit tends to the stable fixed point for any initial value
$x\in[0,1]$. Furthermore, there exists some closed interval $I\In D$,
which depends on $\mu$, containing the stable fixed point such that
$f_\mu(I)\In I$ holds and $f_\mu$ is a contraction on $I$. Next
have a look at the computational effort in the various control
parameter ranges.

\subsection{Numerical analysis of the computational complexity}
The logistic equation is implemented in various forms using an
arbitrary-precision interval library. For that purpose, the already
mentioned interval library MPFI based on the arbitrary-precision
floating-point number library MPFR, both written in C, is used.
For each control parameter $\mu$ ranging from $0.005$ to $4$ and a
step size of $0.005$, the orbit for initial value $x=0.22$ is
calculated up to $N=2000$. For each $\mu$, the minimum
precision $m_{min}$ needed to guarantee $e_n\leq 10^{-6}x_n$
for $n=0,\dots,N$ is searched. Then, $\sigma_{est}:=m_{min}/N$
is calculated. First, $f_\mu$ is implemented using a natural
interval extension based on the expressions $\mu x(1-x)$,
$\mu(x-x^2)$ and $\frac{\mu}{4}-\mu(x-\frac{1}{2})^2$. The natural
interval extension is obtained by replacing any occurrence of the
variable $x$ in the expression by an interval $\ivl{x}$ \cite{rr84}.
The results are shown in Figures \ref{fig_sig_int1}, \ref{fig_sig_int2}
and \ref{fig_sig_int3} respectively. Second, the logistic equation is
implemented using a centered form, actually the mean value form
\cite{rr84, lo93}: $T_{1,\mu}(\ivl{x})=f_\mu(\mathrm{mid}(\ivl{x}))+
\ivl{f}'_\mu(\ivl{x})(\ivl{x}-\mathrm{mid}(\ivl{x}))$ where $\ivl{x}$
is an interval and $\mathrm{mid}(\ivl{x})$ is the midpoint of $\ivl{x}$.
The result is shown in Figure \ref{fig_sig_int4}. In the following,
these 4 calculations are referred to as 1 to 4 respectively.

The interval computation is in agreement with the dynamical picture only in
Calculation 4. While for $\mu\in(0,1)$, the results shown in Calculations
1, 2 and 4 are in agreement with the dynamical analysis, 3 is not since it
would suggest an exponential divergence of initially nearby orbits which is
not true in reality. A similar situation occurs for $\mu\in(1,3)$. Here,
the Calculations 3 and 4 are in agreement with the dynamical picture,
1 and 2 on the other hand not. The picture does not change if
$\mu\geq 3$ and hence it can be said that in the range $\mu\in(1,4]$,
the Calculations 3 and 4 are in agreement with the dynamic picture, while
1 and 2 are not. How can this be explained?

\subsection{Investigating Calculation 1}
This subsection deals with the explanation of the curve obtained by
Calculation 1.
For doing an error analysis of the logistic equation analytically, some
idealizing assumptions have to be made. Generally, executing the iteration in
interval arithmetic, two types of error are present. First, error propagation
solely due to the iteration and second the newly added rounding error caused
by the calculation of $f_\mu$. In the following, only the error propagation
is regarded. This means that there is only one primary made error caused by
rounding the initial value $x=x_0$ to some floating-point number of some specified
precision $m$. The next idealization is that the value of $\mu$ is assumed to
be given with such a high precision that no interval representation is needed.
Finally, the value of $r_n$ in Equation (\ref{err_def}) is neglected. The
recursion relation then reads
\begin{align*}
x^l_{n+1} &= \mu x^l_n(1-x^u_n)\\
x^u_{n+1} &= \mu x^u_n(1-x^l_n)
\end{align*}
with the interval length $d_n$  given by the recursion relation
\begin{align*}
d_{n+1} &= x^u_{n+1} - x^l_{n+1} = \mu(x^u_n - x^u_n x^l_n - x^l_n + x^u_n x^l_n)\\
 &= \mu d_n
\end{align*}
with the obvious solution $d_n=\mu^ n d_0$. The absolute error $e_n$
of $\hat{x}_n$ according to Equation (\ref{val_def}) can be bounded from
above by
\begin{equation}
\label{err_bound}
e_n\leq\frac{1}{2}d_n=\frac{1}{2}\mu^n d_0.
\end{equation}
The ideal assumptions require the somewhat unreal setting
that the precision has to be set to some finite, but big enough value $m$
for representing $x_0$ and a virtually infinite value $m_\infty$ for doing the
iteration. To get a relation connecting $m$ and the output precision $p$ in
(\ref{prec_req}), some upper bound on $d_0$ is needed. The value of $d_0$
is given as the rounding error by representing $x_0$ as a floating-point number
of precision $m$. For that, the well known estimate
\begin{equation}
\label{round_zero}
d_0\leq 2^{-m+1}x_0\leq 2^{-m+1}
\end{equation}
exists. Combining (\ref{prec_req}), (\ref{err_bound}) and (\ref{round_zero})
gives as a sufficient condition
\begin{equation*}
\mu^n\cdot 2^{-m}\leq 10^{-p}
\end{equation*}
for $n=0,\dots,N$. So, the sufficient condition gives an upper bound on
$m_{min}(x,N,p)$ by
\begin{equation*}
m_{min}(x,N,p)\leq \lceil p\cdot\ld(10) + N\cdot\max(0, \ld(\mu))\rceil.
\end{equation*}
This finally leads to an upper bound for the loss of significance rate,
\begin{equation*}
\sigma(x,p)\leq\sigma(x)\leq\max(0, \ld(\mu)).
\end{equation*}

The curve in Figure \ref{fig_sig_int1} shows that $\sigma_{est}$
exceeds the estimated bound $\max(0,\ld(\mu))$ only slightly. So,
the above made ideal assumptions seem to be valid. In \cite{Mue01},
the logistic equation was also investigated for $\mu=3.75$ using
the exact real arithmetic package iRRAM. In the paper, the maximum
bounding precision needed to guarantee the correctness of the first $6$
decimal places are reported up to $N=100000$. Relating this quantity
to $m_{min}$ shows full agreement with the simulation results performed
here. So, for $\mu > 1$, the interval length $d_n$ increases
exponentially in time $n$ which is in contrast to the dynamic behavior
for $\mu\in(1,3)$. The reason is that the natural interval approach
implicitly, due to the dependency problem, takes account only of the
global behavior of $f_\mu$ in the form of a global Lipschitz constant
$\max\{|f'_\mu(x)| : x\in D\}=\mu$. However, a local Lipschitz constant
$\max\{|f'_\mu(x)| : x\in [x^l_n,x^u_n]\}$ governs the real error
propagation at time step $n$ and also describes the dynamic behavior.
\begin{figure}[htb]
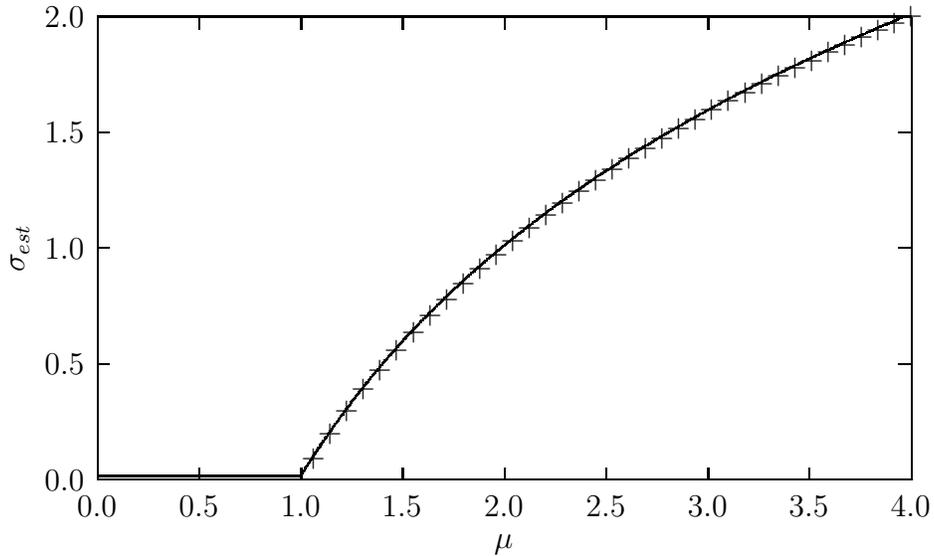

\begin{center}
\include{figure1}
\caption{\label{fig_sig_int1} Estimated loss of significance rate for the
logistic equation, formula $\mu x(1-x)$. The crosses indicate the theoretical
curve from error analysis.}
\end{center}
\end{figure}

\subsection{Investigating Calculation 2}
Calculation 2 is similar to Calculation 1. An analogous analytic
approach as in Calculation 1 gives the recursion relation
\begin{align*}
x^l_{n+1} &= \mu(x^l_n - (x^u_n)^2)\\
x^u_{n+1} &= \mu(x^u_n - (x^l_n)^2)
\end{align*}
and hence
\begin{equation*}
d_{n+1} = \mu d_n + \mu((x^u_n)^2 - (x^l_n)^2)
 = \mu d_n(1+x^u_n+x^l_n).
\end{equation*}

\begin{figure}[htb]
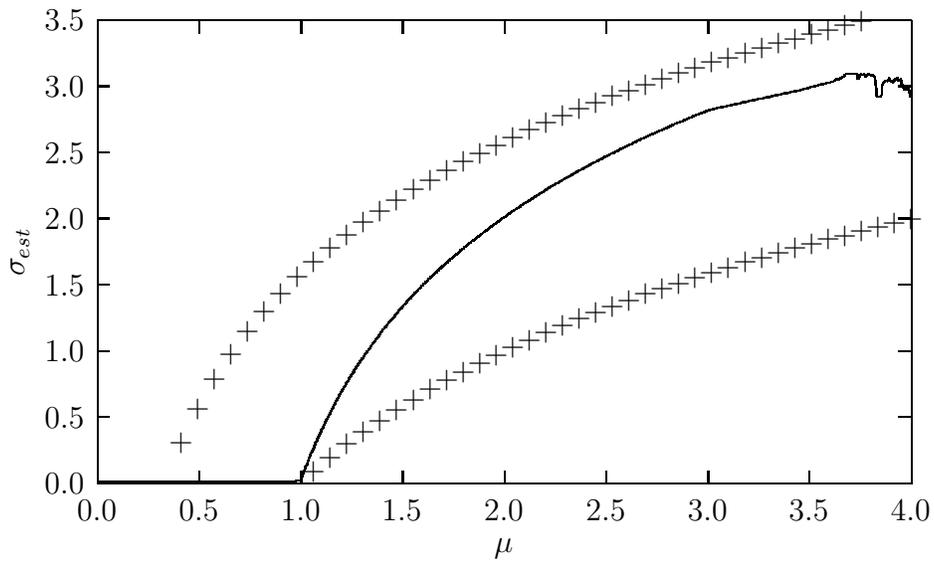

\begin{center}
\include{figure2}
\caption{\label{fig_sig_int2} Estimated loss of significance rate for the
logistic equation, formula $\mu(x-x^2)$. The crosses indicate the theoretical
curves from error analysis.}
\end{center}
\end{figure}
The recursion relation fulfills therefore $d_{n+1}\geq\mu d_n$
and $d_{n+1}\leq 3\mu d_n$. As a consequence the bounds
$\mu^n d_0 \leq d_n\leq (3\mu)^n d_0$ are obtained. In analogy
to Calculation 1, an upper bound for $m_{min}(x,N,p)$,
\begin{equation*}
m_{min}(x,N,p)\leq \lceil p\cdot\ld(10) + N\cdot\max(0, \ld(3\mu))\rceil
\end{equation*}
and hence
\begin{equation*}
\sigma(x,p)\leq\sigma(x)\leq\max(0, \ld(3\mu))
\end{equation*}
is calculated. A brief look at Figure \ref{fig_sig_int2}
shows that this upper bound is too rough. On the other hand,
$d_n\geq \mu d_0$ suggests that the bound derived in Calculation 1
is a lower bound, hence
$\max(0, \ld(\mu))\leq\sigma(x)\leq\max(0, \ld(3\mu))$.
This is actually verified by numerical evidence.

\subsection{Investigating Calculation 3}
Calculation 3 is explained here in the parameter range
$\mu\in(0,1)$, where it is not in agreement with the dynamic
picture. Nevertheless it should be mentioned that the natural interval
extension used here seems to be in full agreement with the dynamic
picture in the parameter range $\mu\in[1,4]$ as is suggested by
Figure \ref{fig_sig_int3}. The curve seems to be identical to
Figure \ref{fig_sig_int4} in the range $\mu\in[1,4]$.

\begin{figure}[htb]
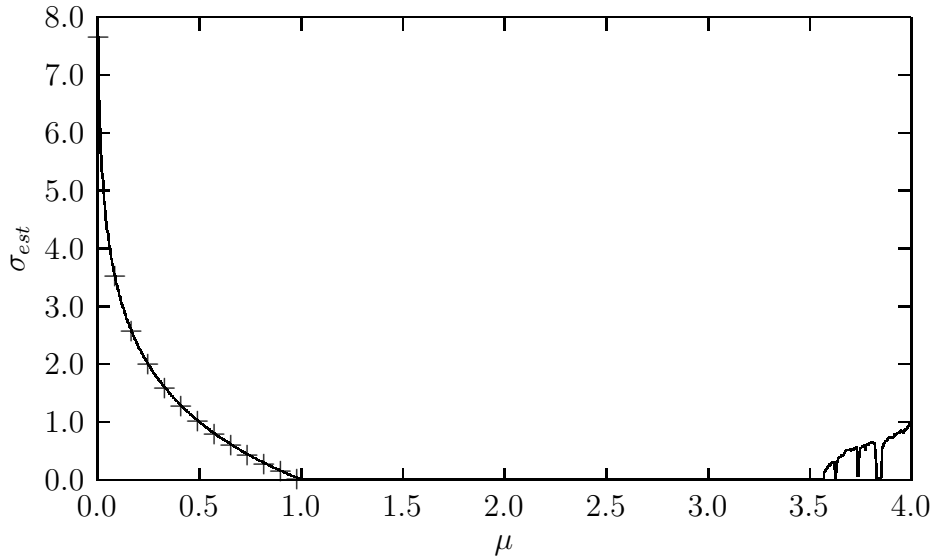

\begin{center}
\include{figure3}
\caption{\label{fig_sig_int3} Estimated loss of significance rate for the
logistic equation, formula $\frac{\mu}{4}-\mu(x-\frac{1}{2})^2$. The
crosses indicate the theoretical curve from error analysis.}
\end{center}
\end{figure}
To explain the observed behavior, first note that for
$\mu<1$, $f_\mu(x)\leq\mu x$ follows for all $x\in[0,1]$. Hence,
$x_n \leq\mu^n x_0$ holds for all $n\in\IN$ and the orbit
$(x_n)_n$ tends exponentially fast to zero. A brief look at the
expression of $f_\mu(x)$,
$f_\mu(x_n)=\frac{\mu}{4}-\mu(x_n-\frac{1}{2})^2$, shows that the
value  of the second term in the difference tends exponentially
fast in $n$ to the value of the first term. Hence, the big values
of the loss of significance rate for small values of $\mu$ can be
explained by cancellation. So, the behavior may be explained
solely by a typical phenomenon of floating-point arithmetic
and not an effect due to the dependency problem in interval
arithmetic. To give an analytical description of the problem,
it is easier now to chance from interval notation to classical
error analysis notation.

First note that even $\hat{x}_0$ may differ from the initial value
$x$ since the conversion to a floating-point number may cause the
very first rounding error. Next, already mentioned, in calculating
the orbit $(\hat{x}_n)_n$, two types of error are present.
First, error propagation due to the iteration scheme and second the
rounding error caused by the calculation of $f_\mu$. Now, let $\hat{x}_n$
for some $n\in\IN$ be given. Then the true error after one iteration step
is $\hat{x}_{n+1} - x_{n+1}$. Since in reality not $f_\mu(\hat{x}_n)$ is
calculated but some erroneous approximation $\hat{f_\mu}(\hat{x}_n)$,
the true error can be written as
$\hat{x}_{n+1} - x_{n+1} = \hat{f}_\mu(\hat{x}_n)-f_\mu(x_n)$.
Inserting a constructive zero gives a sum
\begin{equation}
\label{main_err}
\hat{x}_{n+1} - x_{n+1} = (f_\mu(\hat{x}_n)-f_\mu(x_n)) +
 (\hat{f_\mu}(\hat{x}_n) - f_\mu(\hat{x}_n))
\end{equation}
of two terms. The first term describes solely the error propagation
while the second term gives exactly the newly produced error
due to the approximate calculation of $f_\mu$.

Let us fix some $n\in\IN$ and consider the absolute error of
$ \hat{x}_{n+1}$. To get the formulas more compact, set
$g_\mu(x):=\mu(x-\frac{1}{2})^2$. Then,
\begin{align*}
|\hat{x}_{n+1} - x_{n+1}| &=
|\widehat{(\widehat{(\tfrac{\mu}{4})} - \hat{g}_\mu(\hat{x}_n))} -
(\tfrac{\mu}{4} - g_\mu(x_n))|\\
 &\leq |\widehat{(\widehat{(\tfrac{\mu}{4})} - \hat{g}_\mu(\hat{x}_n))} -
(\widehat{(\tfrac{\mu}{4})} - \hat{g}_\mu(\hat{x}_n))| +
|\widehat{(\tfrac{\mu}{4})} - \tfrac{\mu}{4}|\\
 &\quad + |\hat{g}_\mu(\hat{x}_n) - g_\mu(x_n)|\\
 &\leq |\widehat{(\widehat{(\tfrac{\mu}{4})} - \hat{g}_\mu(\hat{x}_n))} -
(\widehat{(\tfrac{\mu}{4})} - \hat{g}_\mu(\hat{x}_n))| +
|\widehat{(\tfrac{\mu}{4})} - \tfrac{\mu}{4}|\\
 &\quad + |\hat{g}_\mu(\hat{x}_n) - g_\mu(\hat{x}_n)| +
 |g_\mu(\hat{x}_n) - g_\mu(x_n)|
\end{align*}
follows. Let $m$ be assumed to be the actual precision under
calculation at time $n$. The last term in the previous inequality can
be estimated the following way. As discussed in \cite{wi63}, the
rounding error produced in calculating $g_\mu$ can be estimated by
\begin{equation}
\label{est_wil}
|\hat{g}_\mu(\hat{x}_n)-g_\mu(\hat{x}_n)|
\leq 1.06K2^{-m}|g_\mu(\hat{x}_n)|
\end{equation}
where $K$ is the number of rounding operations performed in computing
$\hat{g}_\mu$. In the case considered here, $K=4$ follows. It is
further crucial to mention that the factor $1.06$ is only valid if
$K\leq 0.1\cdot 2^m$ holds so that the precision must not be
chosen too small. Furthermore, with
$|g_\mu(\hat{x}_n) - g_\mu(x_n)|\leq \mu|\hat{x}_n - x_n|$
it follows
\begin{align*}
|\hat{x}_{n+1} - x_{n+1}| &\leq 2^{-m}|\widehat{(\tfrac{\mu}{4})} -
\hat{g}_\mu(\hat{x}_n)| + 2^{-m}\cdot\tfrac{\mu}{4} +
 1.06K2^{-m}|g_\mu(\hat{x}_n)|\\
 &\quad + \mu|\hat{x}_n - x_n|\\
 &\leq 2^{-m}(|\widehat{(\tfrac{\mu}{4})}| + |\hat{g}_\mu(\hat{x}_n)|
 + \tfrac{\mu}{4} + 1.06K|g_\mu(\hat{x}_n)|) + \mu|\hat{x}_n - x_n|\\
 &\leq 2^{-m}((1+2^{-m})\tfrac{\mu}{4} + (1+1.06K2^{-m})
 |g_\mu(\hat{x}_n)| + \tfrac{\mu}{4}\\
 &\quad + 1.06K|g_\mu(\hat{x}_n)|)
 + \mu|\hat{x}_n - x_n|\\
 &\leq 2^{-m}\tfrac{\mu}{4}(1 + 2^{-m} + 1 + 1.06K2^{-m} + 1 + 1.06K)\\
 &\quad + \mu|\hat{x}_n - x_n|\\
 &\leq C\mu 2^{-m} + \mu|\hat{x}_n - x_n| 
\end{align*}
where $C>0$ holds. In other words, one obtains the recursion
relation $e_{n+1}\leq \mu e_n + C\mu 2^{-m}$. Iterating the recursion
gives $e_{n+1}\leq C\mu 2^{-m}\sum_{k=0}^n\mu^k + \mu^{n+1}e_0\leq
C\frac{\mu}{1-\mu}2^{-m} + \mu^{n+1} 2^{-m}x_0$.

As already mentioned, $x_n$ is bounded from above by
$x_n\leq \mu^n x_0$. To come to a sufficient condition
for the precision, also a lower bound is needed. First
observe that $f_\mu(x)\geq\mu x(1-a)$ holds for all $x\leq a$,
$a,x\in[0,1]$. Hence, for $\mu<1$, $x_{n+k}\geq \mu^n x_k(1-x_k)^n$
follows. This gives the sufficient condition
\begin{equation*}
C\tfrac{\mu}{1-\mu}2^{-m} + \mu^{n+1} 2^{-m}x_0
\leq 10^{-p}\mu^{n+1-k}x_k(1-x_k)^{n+1-k}
\leq 10^{-p}x_{n+1}
\end{equation*}
on the precision. Note that $n+1\geq k\geq 0$.
Then, an upper bound on $m_{min}(x_0,N,p)$ is given by
\begin{equation*}
m_{min}(x,N,p)\leq\lceil p\cdot\ld(10) +
(N - k)(\ld(\tfrac{1}{\mu}) - \ld(1-x_k)) + C'\rceil
\end{equation*}
with $C'=\ld(C\frac{\mu}{1-\mu}+\mu^N x_0)-\ld(x_k)$.
This leads to an upper bound on the loss of significance
rate given by $\sigma(x,p)\leq\ld(\frac{1}{\mu})-\ld(1-x_k)$
for all $k\in\IN$. Since $x_k\to 0$ follows for $k\to\infty$,
the final result on the loss of significance rate is
\begin{equation*}
\sigma(x,p)\leq\sigma(x)\leq\ld(\tfrac{1}{\mu}).
\end{equation*}
The curve in Figure \ref{fig_sig_int2} shows that this
upper bound is in full agreement with the numeric result.

\subsection{Investigating Calculation 4}
The observation at the end of the subsection describing Calculation 1
directly leads to the already introduced mean value form.
The calculation is shown in Figure \ref{fig_sig_int4}.
This calculation is the optimum of both, Calculation 1 and 3. The
curve reflects in the parameter range $\mu\in(0,3)$ well the dynamic
behavior.

\begin{figure}[htb]
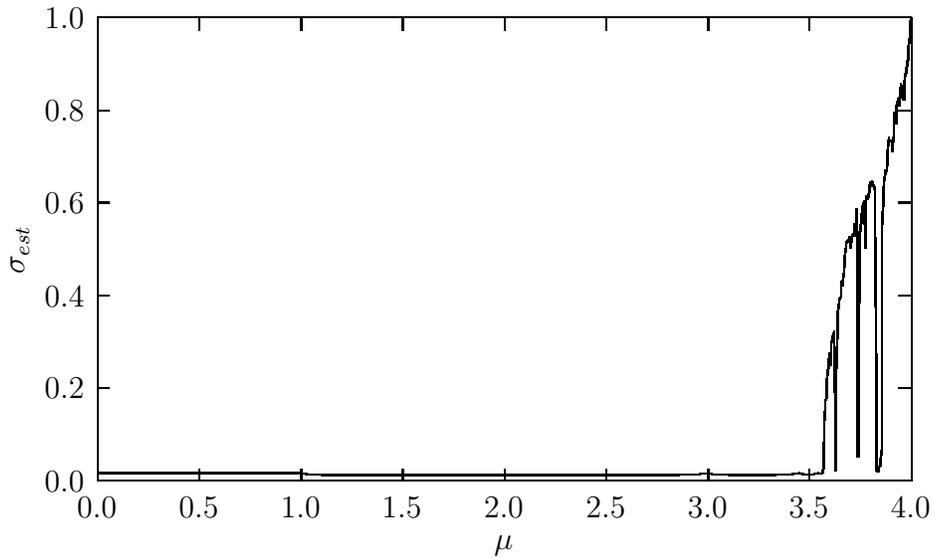

\begin{center}
\include{figure4}
\caption{\label{fig_sig_int4} Estimated loss of significance rate for the
logistic equation, meanvalue form.}
\end{center}
\end{figure}
Furthermore, in the range $\mu\in[3,4]$, the curve suggests
a relation between the loss of significance rate and the Lyapunov
exponent $\lambda(x)$ for the logistic map:
\begin{equation*}
\sigma(x)=\frac{1}{\ln(2)}\max(0,\lambda(x))
\end{equation*}
for all $\mu\in(0,4]$. For a curve of the Lyapunov exponent of the
logistic map see \cite{ce80}. This relation will be shown in the next
section for general dynamical systems on the interval. Furthermore, it
will be shown that the algorithm based on Calculation 4 is optimal in
some sense.

But before, some crucial reflections governing the analysis in the next
section. The mean value form representation, on which the calculation
is based, can also be seen from a different viewpoint. Have again
a look at Equation (\ref{main_err}). The true error is the sum of
the error propagation (first term) according to the iteration and the
rounding error due to the computation of $f_\mu$ (second term). The
first term of Equation (\ref{main_err}) can be handled using the mean
value theorem,
$|f_\mu(\hat{x}_n)-f_\mu(x_n)|=|f'_\mu(y_n)|\cdot |\hat{x}_n - x_n|$
with $y_n\in[\hat{x}_n-e_n,\hat{x}_n+e_n]$. This gives directly the
bound
\begin{equation*}
|f_\mu(\hat{x}_n)-f_\mu(x_n)|\leq
\sup(|f'_\mu([\hat{x}_n-e_n,\hat{x}_n+e_n])|)e_n.
\end{equation*}
The second term can be estimated in a similar way as was done in
(\ref{est_wil}) by
\begin{equation*}
|\hat{f}_\mu(\hat{x}_n)-f_\mu(\hat{x}_n)|\leq
1.06K2^{-m}|f_\mu(\hat{x}_n)|
\end{equation*}
where $K=4$ because there are 3 arithmetic operations and the
rounding of $\mu$. Using the fact that $f_\mu(x)\leq\frac{\mu}{4}$
holds and $f_\mu(x)<x$ if $\mu\leq 1$, the unknown value
$|f_\mu(\hat{x}_n)|$ can be estimated from above.
This calculation shows that there exists a recursive equation
on an upper bound $\overline{e}_n$ on $e_n$ for all $n$:
\begin{equation}
\label{calc_err}
\overline{e}_{n+1}=L(\hat{x}_n,\overline{e}_n)\overline{e}_n+
1.06K2^{-m}E_\mu(\hat{x}_n),\quad \overline{e}_0=2^{-m}
\end{equation}
with $L(x,e):=\sup(|f'_\mu([x-e,x+e])|)$ and
\begin{equation*}
E_\mu(x):=\left\{\begin{array}{ll}
x & {\rm if\ } \mu\leq 1\\
\frac{\mu}{4} & {\rm if\ } \mu >1
\end{array}\right..
\end{equation*}
This description, which is in line with the analysis of
Calculation 3, is equivalent to the interval description
using the mean value form. Instead of using intervals,
pairs of the form value $\hat{x}_n$ and corresponding guaranteed
error bound $\overline{e}_n$ is used. This approach is an
automated error analysis called {\em running error analysis}
\cite{hi02}. From a technical point of view, the representation as
value and error has the advantage that the rounded values $\hat{x}_n$
are calculated as usual in floating-point arithmetic except that
arbitrary-precision floats are used. The guaranteed error bounds
may be calculated using interval arithmetic according to (\ref{calc_err}),
to really guarantee a validated bound. Only a fixed precision is needed
for calculating the error bounds. Similar results as in Figure 
\ref{fig_sig_int4} are reported in \cite{Bla05} by using a method
analog to the one presented here \cite{bl06}. However, the connection
to the Lyapunov exponent is not made in \cite{Bla05}.

Before continuing, three remarks. First, interval libraries are primarily
divided int two types concerning their representation of an interval
\cite{ru99}: There exist libraries using the infimum-supremum representation
of intervals, like MPFI, and there exist libraries using the midpoint-radius
representation of intervals. If arbitrary precision is needed, the inf-sup
libraries have the disadvantage that two floating-point variables with high
precision are needed to represent an interval. Contrary to that, like the
value and error description, in mid-rad libraries only the midpoint of the
interval needs a high precision floating-point variable. The radius can be
stored in a floating-point variable which need not have a high precision.
Clearly, the mid-rad concept has an computational advantage in the case
considered here over the inf-sup concept. But the dependency problem of
interval arithmetic persists.
Second, also the iRRAM package implements mid-rad intervals and has
therefore to cope with the dependency problem. However, it also permits an
optimized way for computing the iteration based on a similar algorithm as
described above \cite{mu05}.
Third it should be mentioned that, executing the first three presented
calculations in Mathemathica using significance arithmetic, exactly the
same results are obtained. This shows that also significance arithmetic
suffers from the dependency problem as interval arithmetic does. This is
already noted in \cite{st74}.

\section{The general algorithm and its complexity}
\label{main-sec}
Let $D$ be a compact real interval and $f:D\to D$ a self mapping. In the following,
$f$ is assumed to be continuous on $D$, two times continuously differentiable on
$D$ and $f''$ is bounded. Furthermore, $f$ and $f'$ are assumed
to be computable in the sense of Computable Analysis. The definition of
a computable real function is given below.

In this section, a general algorithm for computing the iteration
\begin{equation}
\label{main_it}
x_{n+1}=f(x_n),\quad x_0\in D
\end{equation}
is presented. To be more precise, for given $x\in\IQ$, $N\in\IN$ and
$p\in\IZ$, this algorithm computes a finite part $(\hat{x}_n)_{0\leq n\leq N}$
of length $N$ of the true orbit $(x_n)_{n\in\IN}$ with initial value $x_0=x$.
Each computed value $\hat{x}_n$ of this finite trajectory has a relative error
of at most $10^{-p}$: $|\hat{x}_n-x_n|\leq 10^{-p}|x_n|$ for all $n=0,1,\dots,N$.
The correctness of the algorithm and its relation to Computable Analysis is shown.
Finally, its complexity is examined.

\subsection{Computability issues and specifying the algorithm}
The set of all computationally accessible real numbers are the
floating-point numbers of arbitrary precision and arbitrary exponent
range denoted by $\hat{\IR}$. A floating-point number is a real number of
the form $\hat{x}=s\cdot 2^{e-t}$ where $t\in\IN$ is the {\em precision},
$e\in\IZ$ the {\em scale} and $s\in\IZ$ where $|s|\in\{0,1,\dots,2^t -1\}$
is called the {\em significand}. To get a unique representation of $\hat{x}$
for given $t$, $|s|\geq 2^{t-1}$ is assumed if $\hat{x}\neq 0$ and $e=0$ if
$\hat{x}=0$.  Since actually no bound is assumed on the precision and the scale,
the set $\hat{\IR}\In\IR$ is the set of the dyadic real numbers and therefore
countable infinite. Thus, $\hat{\IR}$ forms a natural basis for computability
considerations over finite
objects. Consider some floating-point number $\hat{x}\in\hat{\IR}$, then
the scale and the precision are two properties of different type. While the
scale is a direct function of the value of $\hat{x}$, the precision is clearly
not. Reversely, let $x\in\IR$ be some real number and $\hat{x}\in\hat{\IR}$
a floating-point number representing $x$. Then the scale of $\hat{x}$ is
generally determined by $x$ while the precision can be chosen arbitrary.
Regarding $\hat{x}$ as a data structure, then $\hat{x}$ has as its
essential property the precision. In object oriented notation, the
precision of $\hat{x}$ can be written as $\hat{x}.t$.

Any real number $x$ is represented in an algorithm concerning numerical
computation by a pair $\ivl{x}$ consisting of a floating
point number $\ivl{x}.fl\in\hat{\IR}$ of arbitrary precision
$\ivl{x}.fl.t$ approximating $x$ and a floating-point number
$\ivl{x}.err\in\hat{\IR}$ of
fixed precision giving an upper bound on the absolute error,
$|\ivl{x}.fl-x|\leq\ivl{x}.err$. Reversely, any such pair
$\ivl{x}\in\hat{\IR}^2$ can be seen as
the real interval $[\ivl{x}.fl-\ivl{x}.err,\ivl{x}.fl+\ivl{x}.err]$.
If $x\in[\ivl{x}.fl-\ivl{x}.err,\ivl{x}.fl+\ivl{x}.err]$ holds
for some $x\in\IR$, then $\ivl{x}$ is called an {\em approximation}
of $x$. To represent a single real number, a sequence $(\ivl{x}_n)_{n\in\IN}$
of such pairs $\ivl{x}$ are needed. A sequence $(\ivl{x}_n)_{n\in\IN}$
is called a {\em floating-point name} of a real number $x$, if
any $\ivl{x}_n$ approximates $x$, $\lim_{n\to\infty}\ivl{x}_n.fl=x$,
$\lim_{n\to\infty}\ivl{x}_n.fl.t=\infty$ and
$\lim_{n\to\infty}\ivl{x}_n.err=0$ holds. Clearly any real number has
a floating-point name.

As already indicated, it is a straightforward task to define what
a computable function $\hat{f}:\hat{\IR}\to\hat{\IR}$ is by using
classical computability theory over finite objects. Additionally,
computability over integers, computability of functions with mixed
arguments and computable predicates are defined in the same manner
\cite{Wei87}. Consider a function $f:D\to D$, $D\In\IR$ and a pair
$\ivl{f}$ of two functions $\ivl{f}.fl:\hat{\IR}\to\hat{\IR}$ and
$\ivl{f}.erf:\hat{\IR}^2\to\hat{\IR}$ having the following property.
For any approximation $\ivl{x}$ of some real number $x\in D$, the pair
$\ivl{f}(\ivl{x})=(\ivl{f}.fl(\ivl{x}.fl),\ivl{f}.erf(\ivl{x}))$
is an approximation of $f(x)$. Thus, $\ivl{f}.erf$ gives an
upper bound on the absolute error of $\ivl{f}.fl(\ivl{x}.fl)$,
$|\ivl{f}.fl(\ivl{x}.fl)-f(x)|\leq \ivl{f}.erf(\ivl{x})$. Considering
$\ivl{f}$ as an interval function, the above property is just the
fundamental property of interval arithmetic, \cite{rr88} Property 2.12.
Then, $\ivl{f}$ is called an {\em approximation function} for $f$. Now
consider an approximation
function $\ivl{f}$ for $f$ such that for all $x\in D$ and any floating
point name $(\ivl{x}_n)_{n\in\IN}$ of $x$, $(\ivl{f}(\ivl{x}_n))_{n\in\IN}$
is a floating-point name of $f(x)$. Such an approximation function
is called {\em approximation-continuous}. Additionally, if the two
functions $\ivl{f}.fl$ and $\ivl{f}.erf$ of an approximation function
$\ivl{f}$ are computable, then $\ivl{f}$ is called a
{\em computable approximation function}. Finally, $f:D\to D$ is called
{\em computable}, if there exists a computable approximation function
$\ivl{f}$ for $f$ which is approximation-continuous.

The algorithm with the specification described at the beginning of
this section reads
\begin{tabbing}
{\tt\ 1}\ \= {\tt Input parameter:} $x$, $N$, $p$\\
{\tt\ 2}\> {\tt Initialize precision} $m\leftarrow 1$\\
{\tt\ 3}\> {\tt do} \= \ \\
{\tt\ 4}\> \> {\tt Initialize value and error} $\ivl{x}\leftarrow rd(x,m)$\\
{\tt\ 5}\> \> {\tt for} \= $n=0$ to $N$ {\tt do}\\
{\tt\ 6}\> \> \> {\tt If} \= $prec(\ivl{x},p)={\rm\bf true}$ {\tt then}\\
{\tt\ 7}\> \> \> \> {\tt If not printed print } $n$, $\ivl{x}.fl$, $\ivl{x}.err$ \\
{\tt\ 8}\> \> \> {\tt else break}\\
{\tt\ 9}\> \> \> $\ivl{x}\leftarrow\ivl{f}(\ivl{x})$\\
{\tt 10}\> \> {\tt end for}\\
{\tt 11}\> \> $m\leftarrow m + 1$\\
{\tt 12}\> {\tt while} $prec(\ivl{x},p)={\rm\bf false}$
\end{tabbing}
where $\ivl{f}$ is an approximation-continuous approximation function for $f$
specified below. To initialize $\ivl{x}$, a rounding function
$rd:\IQ\times\IN\to\hat{\IR}^2$ is needed where $rd(x,m).fl$
is a floating-point number of precision $m$ being the exactly rounded
value of $x$ for some rounding convention, in the following nearest. Clearly,
the value $rd(x,m).err$ is an upper bound on the absolute rounding error,
$rd(x,m).err=\frac{1}{2}ulp(x)$ if the rounding mode is nearest.
The predicate $prec:\hat{\IR}^2\times\IZ\to\{{\rm\bf true}, {\rm\bf false}\}$
is a test whether the relative error of $\ivl{x}.fl$ is bounded by $10^{-p}$.
The semantics reads:
\begin{equation}
\label{pred_sem}
\begin{split}
&\text{If } \ivl{x}\in\hat{\IR}^2 \text{ approximates } x\in\IR \text{ and }
prec(\ivl{x},p) = {\rm\bf true} \text{ holds,}\\
&\text{then } |\ivl{x}.fl - x|\leq 10^{-p}|x| \text{ follows.}
\end{split}
\end{equation}

While the object oriented notation is convenient for a compact and
instructive description of the algorithm, in the following analytical
analysis an abbreviation for this notation is sometimes more handsome. As
in the line of the preceding section, floating-point numbers and functions
are indicated by a hat: $\hat{x}:=\ivl{x}.fl$ and $\hat{f}:=\ivl{f}.fl$.
An over-bar indicates an error bound: $\overline{e}:=\ivl{x}.err$ and
$\overline{erf}:=\ivl{f}.erf$. Hence, $\ivl{x}$ is equivalent to
$(\hat{x},\overline{e})$ and $\ivl{f}$ is equivalent to
$(\hat{f},\overline{erf})$.

Finally a remark on optimization. The algorithm
is not optimized in performance. Including performance issues, in
Line {\tt 11} something like $m\leftarrow a\cdot m + b$ can be used
where $a>1$ and $c\in\IN$ are constants. Here, the aim is to
find the minimal $m$ to guarantee some given upper bound on the
relative error of $x_n$.

\subsection{Computability and correctness}
\label{comp-and-corr}
It is clear that the rounding function $rd$ is computable.
So let us begin with the predicate $prec$.
\begin{proposition}
The predicate
\begin{equation}
\label{alg:prec}
prec(\ivl{x},p):=\left\{\begin{array}{ll}
{\rm\bf true} & {\rm if\ }
\ivl{x}.err\leq\frac{10^{-p}}{1+10^{-p}}|\ivl{x}.fl|\\
{\rm\bf false} & {\rm otherwise}
\end{array}\right.
\end{equation}
is computable and satisfies (\ref{pred_sem}).
\end{proposition}
\begin{proof}
Let $\ivl{x}$ be an approximation of $x$. If
$\ivl{x}.err\leq\frac{10^{-p}}{1+10^{-p}}|\ivl{x}.fl|$
holds, then
$\ivl{x}.err\leq 10^{-p}(|\ivl{x}.fl| - \ivl{x}.err)$
follows. Using
$|\ivl{x}.fl|\leq |\ivl{x}.fl - x|+|x| \leq\ivl{x}.err+|x|$,
$|\ivl{x}.fl - x|\leq\ivl{x}.err\leq 
10^{-p}(|\ivl{x}.fl| - \ivl{x}.err)\leq 10^{-p}|x|$
follows.

The predicate (\ref{alg:prec}) only uses the approximation $\ivl{x}$,
basic arithmetic and finite tests. Hence, this formula is
computable.
\end{proof}
Note that the definition of the predicate also gives ${\rm\bf true}$
in the singular case where $\ivl{x}.fl=0$ and $\ivl{x}.err=0$ and
hence $x=0$.

An algorithm for computing $\ivl{f}.fl$ is possible by assumption.
To derive an algorithm for computing $\ivl{f}.erf$ on the absolute
error, return to Equations (\ref{main_err}) and (\ref{calc_err}).
\begin{proposition}
\label{prop:ubound}
Let $x\in D$ be given and $\ivl{x}$ an approximation of $x$ with
$\ivl{x}.fl\in D$. Assume that $\ivl{f}.fl(\ivl{x}.fl)$ computes
the value $f(\ivl{x}.fl)$ up to a correctly rounded last bit in
the significand.\footnote{This assumption is pragmatic. The already
mentioned software package MPFR implements this specification.
The problem of achieving this task for transcendental functions
may be of unknown cost and is known as The Table Maker's Dilemma, see
\url{http://perso.ens-lyon.fr/jean-michel.muller/Intro-to-TMD.htm}.
Additionally note that this assumption can be weakened without
abandoning the main statements of this work.}
Furthermore assume $\ivl{f}.fl(\ivl{x}.fl).t=\ivl{x}.fl.t$.
Then the absolute error of $\ivl{f}(\ivl{x})$ is
bounded from above by
\begin{equation}
L(\ivl{x})\cdot\ivl{x}.err + 2^{-\ivl{x}.fl.t}\cdot|\ivl{f}.fl(\ivl{x}.fl)|.
\end{equation}
Here, $L(\ivl{x})=\sup(|f'([\ivl{x}.fl-\ivl{x}.err,\ivl{x}.fl+\ivl{x}.err]
\cap D)|)$.
\end{proposition}
\begin{proof}
Equation (\ref{main_err}) gives $|\hat{f}(\hat{x})-f(x)|\leq
|f(\hat{x})-f(x)|+|\hat{f}(\hat{x})-f(\hat{x})|$. Using the mean
value theorem, $|f(\hat{x})-f(x)|\leq
\sup(|f'([\hat{x}-\overline{e},\hat{x}+\overline{e}])|)\overline{e}$
follows. According to the assumption on $\hat{f}$ and Theorem 2.3
of \cite{hi02},
$|\hat{f}(\hat{x})-f(\hat{x})|\leq 2^{-m}|\hat{f}(\hat{x})|$
holds where $m$ is the precision of $\hat{x}$.
\end{proof}
\begin{corollary}
\label{cor:encf}
Let $f$ be as specified in the beginning of this section,
$\ivl{f}.fl$ and $L(\ivl{x})$ specified as in Proposition
\ref{prop:ubound}. Then there exists a function
$\overline{L}(\ivl{x})$ with
$L(\ivl{x})\leq\overline{L}(\ivl{x})\leq \overline{L}_{max}$ for
some $\overline{L}_{max}\geq 0$ such that $\ivl{f}$ with
$\ivl{f}.erf(\ivl{x})=\overline{L}(\ivl{x})\cdot\ivl{x}.err+
2^{-\ivl{x}.fl.t}\cdot|\ivl{f}.fl(\ivl{x}.fl)|$ is an
approximation-continuous, computable approximation
function of $f$.
\begin{equation*}
\end{equation*}
\end{corollary}
\begin{proof}
Let $\overline{L}(\ivl{x})$ be some computable upper bound of
$L(\ivl{x})$. $\overline{L}(\ivl{x})$ can be computed by
global optimization, for example by using interval arithmetic.
Since $f'$ is continuous and $D$ compact, $L(\ivl{x})$ is
bounded. So, $\overline{L}(\ivl{x})\leq \overline{L}_{max}$ for some
$\overline{L}_{max}\geq 0$. Also, $\ivl{f}$ is computable. Using
Proposition \ref{prop:ubound}, it follows that $\ivl{f}$ is
also an approximation function of $f$. Remains to show that
$\ivl{f}$ is approximation-continuous. Let $(\ivl{x}_n)_n$ be
some floating-point name of $x\in D$. Clearly
$\lim_{n\to\infty}\ivl{f}.fl(\ivl{x}_n.fl).t=\infty$ holds.
Since $\lim_{n\to\infty}\ivl{x}_n.err=0$ and the sequences
$(\overline{L}(\ivl{x}_n))_n$ and $(|\ivl{f}.fl(\ivl{x}_n.fl)|)_n$
are bounded, $\lim_{n\to\infty}\ivl{f}.erf(\ivl{x}_n)=0$
follows. Furthermore, by this result and the statement of
Proposition \ref{prop:ubound},
also $\lim_{n\to\infty}\ivl{f}.fl(\ivl{x}_n)=x$ holds.
\end{proof}

To summarize, the iteration (\ref{main_it}) is performed
in the algorithm by iterating a value $\hat{x}_n$ approximating $x_n$
with an upper bound on its absolute error $\overline{e}_n$ according to
\begin{align}
\label{comp_rec1}
\hat{x}_{n+1} & =\hat{f}(\hat{x}_n) & \hat{x}_0 & =rd(x,m)\\
\label{comp_rec2}
\overline{e}_{n+1} & =\overline{L}(\hat{x}_n,\overline{e}_n)\overline{e}_n +
2^{-m}|\hat{x}_{n+1}| &
 \overline{e}_0 & =2^{-m}|\hat{x}_0|
\end{align}
where $\overline{L}(\hat{x}_n,\overline{e}_n)$ is a computable
upper bound on $L(\hat{x}_n,\overline{e}_n)$ as described in the
preceding corollary and $m$ the precision of any floating-point
number involved at that stage.
This is Line {\tt 9} in the inner {\tt for}-loop of the algorithm
which is executed with successively increasing precision $m$,
controlled by the outer {\tt do-while}-loop. Finally, it has to
be shown that this outer loop eventually terminates.
\begin{proposition}
Let $x\in D$ with $x\neq 0$ be given and $(\ivl{x}_m)_{m\geq 1}$ a
floating-point name of $x$ obeying  $\ivl{x}_m.fl.t = m$.
Then $\lim_{m\to\infty}prec(\ivl{x}_m,p)={\bf true}$ follows
for all $p\in\IZ$.
\end{proposition}
\begin{proof}
Since $x\neq 0$ and $\lim_{m\to\infty}\ivl{x}_m.err=0$,
there exists some $M\in\IN$ such that for all $m\geq M$,
$\frac{1}{2}|x|\leq|\ivl{x}_m.fl|$ and
$\ivl{x}_m.err\leq\frac{10^{-p}}{2(1+10^{-p})}|x|$ holds for
all $m\geq M$. Then $prec(\ivl{x}_m,p)={\bf true}$ for all
$m\geq M$.
\end{proof}
The next proposition  makes the link to Line {\tt 9}
in the algorithm.
\begin{proposition}
\label{prop:alg_halt}
Let $x_n$ be the $n$-th element of the orbit of the recursion
(\ref{main_it}) and $((\ivl{x}_n)_m)_{m\geq 1}$ a sequence given
according to  the recursion equations (\ref{comp_rec1}) and
(\ref{comp_rec2}) with increasing precision
$(\ivl{x}_n)_m.fl.t = m$. Then $((\ivl{x}_n)_m)_{m\geq 1}$ is
a floating-point name of $x_n$.
\end{proposition}
\begin{proof}
Let  $\overline{L}_{max}$ according to Corollary \ref{cor:encf}
and $\overline{M}\geq\sup\{|x| : x\in D\}$ such that
$|\hat{x}_n|\leq\overline{M}$ holds for all $n$.
Then Equation (\ref{comp_rec2}) leads to 
$\overline{e}_{n+1}\leq \overline{L}_{max}\overline{e}_n +
2^{-m}\overline{M}$. Iteration gives
$\overline{e}_n\leq\overline{L}^n_{max}\overline{e}_0+
2^{-m}\overline{M}\sum^{n-1}_{k=0}\overline{L}_{max}^k \leq
2^{-m}\overline{M} \sum^n_{k=0}\overline{L}_{max}^k$.
Hence, for $n$ fixed, $\lim_{m\to\infty}(\ivl{x}_n)_m.err=0$
follows and consequently also
$\lim_{m\to\infty}(\ivl{x}_n)_m.fl=x_n$.
\end{proof}
These two propositions finish the correctness proof of the
algorithm. They show that, if $x_n\neq 0$ for $n=0,\dots,N$,
the outer loop eventually terminates for any $p\in\IZ$.

The drawback of the algorithm is, that in the case $x_n=0$
for some $n\leq N$, the computation does not terminate. This
is only due to the fact that the relative error controls
the outer {\tt do-while}-loop. If the absolute error would be
used instead, this drawback is eliminated. However, controlling
the relative error is more general. Consider for example a
dynamics with positive phase space, the concentration of a
substance for example. If the value varies in time over a
wide range in scale, it is fortunate to illustrate the
orbit in a logarithmic plot. If the relative error
is controlled, the error bars in the plot are constant,
in contrast to large varying error bars in the case where
absolute errors are used.

Absolute errors are in the line with Computable Analysis.
Replacing the test $prec(\ivl{x},p)$ by the test on
$\ivl{x}.err \leq 10^{-p}$ in the algorithm would give
a segment $(x_n)_{0\leq n\leq N}$ of the orbit with
accuracy $|\hat{x}_n - x_n|\leq 10^{-p}$. It is now
straightforward to see that the function $g:D\times\IN\to D$
with $g(x,n):=f^n(x)$ is computable. Here, a function
$g:D\times\IN\to D$ is computable if there exists a
computable approximation function
$\ivl{g}:\hat{\IR}^2\times\IN\to\hat{\IR}^2$ for
$g$ which is approximation-continuous with respect to the
first argument.

\subsection{Computational complexity}
After having presented the preliminary work, the main issue of the
paper is addressed - the computational complexity of the presented
algorithm. The complexity measure of interest here is the loss of
significance rate already introduced informally in the previous section.
Here is the formal definition.
\begin{definition}
\label{def:losr}
The minimal precision, for which the described algorithm
eventually halts is denoted by $m_{min}(x,N,p)$, where $x$, $N$
and $p$ are the corresponding input parameters. The growth
rate of $m_{min}$ is given by
\begin{equation}
\label{main_gr}
\sigma(x,p):=\limsup_{N\to\infty}\frac{m_{min}(x,N,p)}{N}.
\end{equation}
Then, the {\em loss of significance rate}
$\sigma:\IQ\cap D\to\IR$ is defined by
\begin{equation}
\label{main_lsr}
\sigma(x):=\lim_{p\to\infty}\sigma(x,p).
\end{equation}
\end{definition}
To achieve bounds on the loss of significance rate,
the drawback of the preceding subsection also makes problems
here. If $x_n=0$ for some $n\in\IN$, the loss of significance
rate may be unbounded. Therefore, one more assumption in addition
to the ones on the dynamical system stated in the beginning of this
section has to be made.
\begin{assumption}
\label{ass:main}
The dynamical system $(D,f)$ is assumed to have the properties
already mentioned in the beginning of this section and
furthermore, for any orbit $(x_n)_n$ under consideration,
$x_n\neq 0$ holds for any $n\in\IN$ as well as
\begin{equation*}
\lim_{N\to\infty}\frac{\ld(\min\{|x_n| : n=0,1,\dots,N\})}{N}=0.
\end{equation*}
\end{assumption}
If only a finite range in scale is relevant, the additional
assumption is no loss of generality. An example is the logistic
equation where $0\in D$ but $0$ has no distinguished role.
Instead of considering $(D,f)$, consider the following dynamical
system $(\tilde{D},\tilde{f})$. Choose some $M>-\min(D)$ and
set $\tilde{D}:=\{x+M \mid x\in D\}$ as well as
$\tilde{f}(x):=f(x-M)+M$ for all $x\in\tilde{D}$. Then
$(\tilde{D},\tilde{f})$ fulfills the additional assumption.
Furthermore $\tilde{f}'(x)=f'(x-M)$ holds and therefore there
is no substantial difference in the complexity analysis of the
algorithm between the original system and the modified system.

First, the boundedness of $\sigma(x)$ is shown.
\begin{proposition}
Let $(D,f)$ be as in Assumption \ref{ass:main} and $m_{min}(x,N,p)$
as in Definition \ref{def:losr}. Then, for given $p\in\IZ$, there exists
a constant $C\geq 0$, depending on $f$, such that
$m_{min}(x,N,p)\leq C\cdot N + o(N)$ holds for all
$N\in\IN$, $x\in\IQ\cap D$.
\end{proposition}
\begin{proof}
According to the requirements made on $(D,f)$, there are some constants
$L>1$ and $M>0$ such that $\overline{e}_{n+1}\leq L\overline{e}_n+2^{-m} M$
holds for all $n\in\IN$ and all precisions $m$.
Analogous to the treatment in the proof of Proposition
\ref{prop:alg_halt}, iteration gives
$\overline{e}_n\leq 2^{-m}M\sum^n_{k=0}L^k=2^{-m}M\frac{L^{n+1}-1}{L-1}$.
Let $B(N):=\min\{|\hat{x}_n|:n=0,1,\dots,N\}> 0$. Then, for all $n\leq N$,
$\overline{e}_n/|\hat{x}_n|\leq\overline{e}_n/B(N)\leq
\frac{M}{(L-1)g(N)}2^{-m}L^{N+1}$ follows. If now
$\frac{M}{(L-1)B(N)}2^{-m}L^{N+1}\leq \frac{10^{-p}}{1+10^{-p}}$ holds,
$prec((\hat{x}_n,\overline{e}_n),p)={\rm\bf true}$ for all
$n=0,\dots,N$. This leads to the bound
$m_{min}(x,N,p)\leq\ld(L)\cdot N +
\max(1, \ld(\frac{LM}{L-1}) - \ld(B(N)) + p\cdot\ld(10) + \ld(1+10^{-p}))$.
\end{proof}
\begin{corollary}
Let $(D,f)$ be as in Assumption \ref{ass:main}, $\sigma(x,p)$ as in
(\ref{main_gr}) and $\sigma(x)$ the loss of significance rate. Then, for
given $p\in\IZ$, there exists some constant $C\geq 0$ such that
$\sigma(x,p)\leq\sigma(x)\leq C$ holds for all $x\in\IQ\cap D$.
\end{corollary}
In the following, the main statements of this paper are be formulated:
A lower and an upper bound for the loss of significance rate is given.
Furthermore, the relation of these bounds to the Lyapunov exponent
$\lambda(x)$ is shown. Before the theorem is stated, for sake of
completeness, the definition of the Lyapunov exponent and its basic
properties are presented.
\begin{definition}
\label{def:ljap}
Let $(D,f)$ be a dynamical system, $D\In\IR$ compact and $f:D\to D$
continuously differentiable on $D$. Then the
{\em Lyapunov exponent} at $x\in D$ is defined by
\begin{equation}
\label{ljap_def}
\lambda(x):=\lim_{n\to\infty}\frac{1}{n}\sum^{n-1}_{k=0}
\ln(|f'(f^k(x))|)
\end{equation}
if the limit exists.
\end{definition}
The Lyapunov exponent may depend on $x$. However, the following
properties hold:
\begin{itemize}
\item[(a)]
If $(D,f)$ has an {\em invariant measure} $\rho$, then the limit in
Equation (\ref{ljap_def}) exists $\rho$-almost everywhere.
\item[(b)]
Furthermore, if $\rho$ is {\em ergodic} then $\lambda(x)$
is $\rho$-almost everywhere constant and equal to
\begin{equation*}
\int_D \ln(|f'(x)|)\,\rho(dx).
\end{equation*}
\end{itemize}
These properties are a direct consequence of the Birkhoff ergodic theorem,
see \cite{kh95}, Theorem 4.1.2 and Corollary 4.1.9. Now the first theorem.
\begin{theorem}
Let $(D,f)$ be as in Assumption \ref{ass:main}, $\sigma(x,p)$ as in
(\ref{main_gr}) and $\lambda(x)$ the Lyapunov exponent of
$(D,f)$. Then $\sigma(x,p)\geq\max(0,\lambda(x))/\ln(2)$ holds for all
$x\in\IQ\cap D$, $p\in\IZ$ if $\lambda(x)$ exists.
\end{theorem}
\begin{proof}
Let $N\in\IN$ be given and $M > 0$ a constant with $|\hat{x}_n|\leq M$
for all $n\in\IN$. According to Equation (\ref{comp_rec2}) and Proposition
\ref{prop:ubound}, $\overline{e}_{n+1}\geq|f'(x_n)|\overline{e}_n$
holds. Iteration gives
$\overline{e}_N\geq |\hat{x}_0|2^{-m}\prod^{N-1}_{n=0}|f'(x_n)|$.
So, $\frac{\overline{e}_N}{|\hat{x}_N|}\geq
\frac{|\hat{x}_0|2^{-m}}{M}\prod^{N-1}_{n=0}|f'(x_n)|$ follows.
A necessary condition for the algorithm to terminate is therefore
$\frac{|\hat{x}_0|}{M}2^{-m}\prod^{N-1}_{n=0}|f'(x_n)|\leq
\frac{10^{-p}}{1+10^{-p}}$. This gives the bound on
$m_{min}(x,N,p)\geq\sum^{N-1}_{n=0}\ld(|f'(x_k)|)+\ld(\frac{|\hat{x}_0|}{M})+
p\cdot\ld(10)+\ld(1+10^{-p})$. Following the definitions of $\sigma(x,p)$
and the Lyapunov exponent, $\sigma(x,p)\geq\lambda(x)/\ln(2)$ follows.
\end{proof}

Before a realistic upper bound on $\sigma(x,p)$ can be presented, one more
definition is needed.
\begin{definition}
Let $\alpha>0$ then define a function $\eta_\alpha:(0,\infty)\to\IR$ by
\begin{equation*}
\eta_\alpha(x):=\left\{\begin{array}{ll}
\ln(x) & {\rm if\ } x\geq\alpha\\
\ln(\alpha) & {\rm if\ } x<\alpha
\end{array}\right..
\end{equation*}
Furthermore, for any $\alpha>0$ define
\begin{equation*}
\overline{\lambda}_\alpha(x):=\limsup_{n\to\infty}\frac{1}{n}
\sum^{n-1}_{k=0} \eta_\alpha(|f'(f^k(x))|)
\end{equation*}
\end{definition}
\begin{proposition}
For all $\alpha>0$ there exists some constant $C\geq 0$
such that $\overline{\lambda}_\alpha(x)\leq C$
holds for all $x\in D$. Furthermore, if the Lyapunov
exponent $\lambda(x)$ exists,
$\lambda(x)\leq\overline{\lambda}_\alpha(x)$ holds.
\end{proposition}
\begin{proof}
According to the requirements made on $(D,f)$, $f$ is Lipschitz
with a Lipschitz constant $L>0$. Furthermore, let $\alpha>0$
be given. Then for all $n\in\IN$,
$\frac{1}{n}\sum^{n-1}_{k=0}\eta_\alpha(|f'(f^k(x))|)\leq
\ln(\max(\alpha,L))$ holds. Hence it follows the upper
bound on $\limsup_{n\to\infty}\frac{1}{n}\sum^{n-1}_{k=0}
\eta_\alpha(|f'(f^k(x))|)\leq\ln(\max(\alpha,L))$.
The second assertion follows from the
fact that $\ln(x)\leq\eta_\alpha(x)$ holds for all $x>0$,
$\alpha>0$. 
\end{proof}
\begin{proposition}
Let $x\in D$ be given. If $\lambda(x)$ exists, then
also the limit
\begin{equation}
\label{ljap_sup}
\lim_{\substack{\alpha\to 0\\ \alpha>0}}\overline{\lambda}_\alpha(x)=:
\overline{\lambda}(x)
\end{equation}
exists and $\overline{\lambda}(x)\geq\lambda(x)$.
\end{proposition}
\begin{proof}
Since $\ln(x)\leq\eta_\alpha(x)\leq\eta_\beta(x)$ holds
for all $x>0$, $0<\alpha\leq\beta$, also
$\lambda(x)\leq\overline{\lambda}_\alpha(x)\leq
\overline{\lambda}_\beta(x)$ follows. Letting
$\alpha\to 0$, $\alpha > 0$, the assertion follows.
\end{proof}
\begin{theorem}
\label{thm:ubnd}
Let $(D,f)$ be as in Assumption \ref{ass:main}, $\sigma(x,p)$ as in
(\ref{main_gr}) and $\overline{\lambda}(x)$ as in
(\ref{ljap_sup}). Let $x\in\IQ\cap D$ be given, then for any
$\varepsilon>0$ there is some $p_0\in\IZ$ such that for all $p\geq p_0$,
\begin{equation*}
\sigma(x,p)\leq\tfrac{1}{\ln(2)}\max(0,\overline{\lambda}(x))
+\varepsilon
\end{equation*}
holds if $\lambda(x)$ exists.
\end{theorem}
So there is the following bound on the loss of significance rate.
\begin{corollary}
\label{main_stat1}
Let $(D,f)$ be as in Assumption \ref{ass:main}, $\sigma(x,p)$ as in
(\ref{main_gr}), $\sigma(x)$ the loss of significance rate,
$\overline{\lambda}(x)$ as in (\ref{ljap_sup}) and $\lambda(x)$ the
Lyapunov exponent. Then,
\begin{equation*}
\tfrac{1}{\ln(2)}\max(0,\lambda(x))\leq \sigma(x)\leq
\tfrac{1}{\ln(2)}\max(0,\overline{\lambda}(x))
\end{equation*}
holds for all $x\in\IQ\cap D$ if $\lambda(x)$ exists.
\end{corollary}

Before the proof of the theorem can be presented, the following lemma
is needed.
\begin{lemma}
Let $\varepsilon\geq 0$ and $\alpha>\sqrt{\varepsilon}$.
Then for all $x>0$,
\begin{equation*}
\ln(x+\varepsilon)\leq\eta_\alpha(x)+\sqrt{\varepsilon}
\end{equation*}
holds.
\end{lemma}
\begin{proof}
There is nothing to prove in the case $\varepsilon=0$. So
let $\varepsilon>0$. Two cases are considered.
 
1st case: $x\geq\alpha$.
Then the inequality reads
$\ln(x+\varepsilon)\leq\ln(x)+\sqrt{\varepsilon}$
which is equivalent to
$x\geq\frac{\varepsilon}{\exp(\sqrt{\varepsilon})-1}$.
Since $\frac{\varepsilon}{\exp(\sqrt{\varepsilon})-1}\leq
\frac{\varepsilon}{\sqrt{\varepsilon}}<\alpha\leq x$,
the assertion follows.

2nd case: $x<\alpha$.
Then the inequality reads
$\ln(x+\varepsilon)\leq\ln(\alpha)+\sqrt{\varepsilon}$
which is equivalent to
$x\leq\alpha\exp(\sqrt{\varepsilon})-\varepsilon$.
A sufficient condition to prove the assertion is
$\alpha\leq\alpha\exp(\sqrt{\varepsilon})-\varepsilon$
which is equivalent to
$\alpha\geq\frac{\varepsilon}{\exp(\sqrt{\varepsilon})-1}$.
This was already proven in the first case.
\end{proof}
Now everything is prepared to prove Theorem \ref{thm:ubnd}.
\begin{proof}[Proof of Theorem \ref{thm:ubnd}]
Let $N\in\IN$, $B(N):=\min\{|\hat{x}_n|:n=0,1,\dots,N\}> 0$
and $M > 0$ a constant with $|\hat{x}_n|\leq M$ for all $n\in\IN$.
Starting with Equation (\ref{comp_rec2}) and iterating gives
\begin{align*}
\overline{e}_n & =\overline{e}_0\prod^{n-1}_{l=0}
\overline{L}(\hat{x}_l,\overline{e}_l) + 2^{-m}\sum^{n}_{k=1}
|\hat{x}_k|\prod^{n-1}_{l=k}\overline{L}(\hat{x}_l,\overline{e}_l)\\
 & = 2^{-m}\sum^{n}_{k=0}|\hat{x}_k|\prod^{n-1}_{l=k}
 \overline{L}(\hat{x}_l,\overline{e}_l)
 \leq M2^{-m}\sum^{n}_{k=0}\prod^{n-1}_{l=k}
 \overline{L}(\hat{x}_l,\overline{e}_l).
\end{align*}
Define
\begin{equation*}
sp_n:=\sum^{n}_{k=0}\prod^{n-1}_{l=k}
\overline{L}(\hat{x}_l,\overline{e}_l)
\end{equation*}
and
\begin{equation*}
SP(N):=\max\{sp_n : n=0,1,\dots,N\}.
\end{equation*}
Then, $\overline{e}_n\leq M2^{-m}SP(N)$ follows for all $n\leq N$.
A sufficient condition for the algorithm to
terminate is given by
$\frac{M}{B(N)}2^{-m}SP(N)\leq\frac{10^{-p}}{1+10^{-p}}$. Hence,
\begin{equation*}
m_{min}(x,N,p)\leq \max(1,C-\ld(B(N))+\ld(SP(N)))
\end{equation*}
follows with
$C=\ld(M)+p\cdot\ld(10)+\ld(1+10^{-p})$. Using the Assumption
\ref{ass:main} leads to
\begin{equation*}
\sigma(x,p)\leq \frac{1}{\ln(2)}\max\left(0,
\limsup_{N\to\infty}\frac{1}{N}\ln(SP(N))\right).
\end{equation*}
By definition, $\limsup_{N\to\infty}\frac{\ln(SP(N))}{N}\leq
\limsup_{n\to\infty}\frac{\ln(sp_n)}{n}$ follows and hence
\begin{equation*}
\sigma(x,p)\leq \frac{1}{\ln(2)}\max\left(0,
\limsup_{n\to\infty}\frac{\ln(sp_n)}{n}\right).
\end{equation*}
Next let
\begin{equation*}
p_{k,n}:= \prod^{n-1}_{l=k}\overline{L}(\hat{x}_l,\overline{e}_l)
\end{equation*}
for $n\in\IN$ and $k\leq n$, and furthermore
\begin{equation*}
P_n:=\max\{p_{k,n} : k=0,1,\dots,n\},
\end{equation*}
then $sp_n\leq (n+1)P_n$ follows for all $n\in\IN$. This gives
\begin{equation*}
\sigma(x,p)\leq \frac{1}{\ln(2)}\max\left(0,
\limsup_{n\to\infty}\frac{\ln(P_n)}{n}\right).
\end{equation*}
Let $K(n)\in\{0,\dots,n\}$ be the smallest number such that
$\prod^{n-1}_{l=K(n)}\overline{L}(\hat{x}_l,\overline{e}_l)=P_n$.
Then consider $\frac{\ln(P_n)}{n}=
\frac{1}{n}\sum^{n-1}_{l=K(n)}\ln(\overline{L}(\hat{x}_l,\overline{e}_l))$.
Let $L'$ be a Lipschitz constant of
$f'$, then $L(\hat{x}_n,\overline{e}_n)\leq|f'(x_n)|+
L'2\overline{e}_n$ holds for all $n\in\IN$. Consequently,
there exists some $\overline{L}'\geq L'$ such that
$\overline{L}(\hat{x}_n,\overline{e}_n)\leq|f'(x_n)|+
2\overline{L}'\overline{e}_n$ holds for all $n\in\IN$.
This inequality leads to
$\overline{L}(\hat{x}_n,\overline{e}_n)\leq|f'(x_n)|+
2\overline{L}'M\frac{10^{-p}}{1+10^{-p}}\leq
|f'(x_n)|+2\overline{L}'M\cdot10^{-p}$.
Inserting gives
\begin{equation*}
\frac{\ln(P_n)}{n}\leq\frac{1}{n}\sum^{n-1}_{l=K(n)}\ln(|f'(x_l)|+
2\overline{L}'M\cdot 10^{-p}).
\end{equation*}

Now let $\varepsilon>0$ and $0<\alpha<1$ be given. Then choose
$p_0\in\IN$ such that $\sqrt{2\overline{L}'M\cdot 10^{-p_0}} <
\min(\alpha,\ln(2)\frac{\varepsilon}{2})$ holds. Then for all
$p\geq p_0$, the above lemma gives
\begin{align}
\frac{\ln(P_n)}{n} &\leq \frac{1}{n}\sum^{n-1}_{l=K(n)}
\left(\eta_\alpha(|f'(x_l)|)+\ln(2)\frac{\varepsilon}{2}\right)\\
\label{main_PE}
 &\leq \ln(2)\frac{\varepsilon}{2} + 
 \frac{1}{n}\sum^{n-1}_{l=K(n)}\eta_\alpha(|f'(x_l)|).
\end{align}
Consider the sequence $(K(n))_{n\in\IN}$. Observe that, first the
sequence $(K(n))_{n\in\IN}$ is increasing and second if $K(n+1)>K(n)$
for some $n\in\IN$, then $K(n+1)=n$ or $K(n+1)=n+1$. There are two cases.

1st case: $(K(n))_{n\in\IN}$ is bounded. Then, there exists
some constant $N_0\in\IN$ such that $K(n)=K(N_0)$ holds for all $n\geq N_0$.
Choose now $\alpha$ small enough such that $\overline{\lambda}_\alpha(x)\leq
\overline{\lambda}(x)+\ln(2)\frac{\varepsilon}{2}$ holds.
Then, compute the upper limit to $\limsup_{n\to\infty}\frac{1}{n}
\sum^{n-1}_{l=K(n)}\eta_\alpha(|f'(x_l)|) = \limsup_{n\to\infty}\frac{1}{n}
\sum^{n-1}_{l=K(N_0)}\eta_\alpha(|f'(x_l)|) = \limsup_{n\to\infty}\frac{1}{n}
\sum^{n-1}_{l=0}\eta_\alpha(|f'(x_l)|) = \overline{\lambda}_\alpha(x)$.
By taking the upper limit of (\ref{main_PE}),
$\limsup_{n\to\infty}\frac{\ln(P_n)}{n}\leq \ln(2)\frac{\varepsilon}{2} +
\overline{\lambda}_\alpha(x) \leq \ln(2)\varepsilon + \overline{\lambda}(x)$
follows.

2nd case: $(K(n))_{n\in\IN}$ is not bounded. Then, for any
$\delta > 0$ and any $N_0\in\IN$ there is some $n\geq N_0$
with $\frac{\ln(P_n)}{n}<\delta$. Since, by definition,
$\sum^{n-1}_{l=0}\ln(\overline{L}(\hat{x}_l,\overline{e}_l))\leq
\ln(P_n)$ holds as well as
$|f'(x_l)|\leq\overline{L}(\hat{x}_l,\overline{e}_l)$, the
inequality
$\frac{1}{n}\sum^{n-1}_{l=0}\ln(|f'(x_l)|)\leq\frac{\ln(P_n)}{n}$
follows. This shows $\lambda(x)\leq 0$.

Next it is stated that for all $\varepsilon >0$ and $p\geq p_0$,
$\limsup_{n\to\infty}\frac{\ln(P_n)}{n}\leq \ln(2)\varepsilon$
holds. This shows $\sigma(x,p)\leq\varepsilon=\frac{1}{\ln(2)}
\max(0,\lambda(x))+\varepsilon\leq\frac{1}{\ln(2)}
\max(0,\overline{\lambda}(x))+\varepsilon$.

Assume otherwise. Then, for some $\varepsilon>0$ and $N\in\IN$,
first $\frac{\ln(P_N)}{N}> \ln(2)\varepsilon$ holds and second
$\lambda(x)-\ln(2)\frac{\varepsilon}{4}<
\frac{1}{n}\sum_{l=0}^{n-1}\ln(|f'(x_l)|)<
\lambda(x)+\ln(2)\frac{\varepsilon}{4}$ holds for all
$n\geq K(N)$. Using (\ref{main_PE}), the first expression
gets $\frac{1}{N}\sum^{N-1}_{l=K(N)}\eta_\alpha(|f'(x_l)|)>
\ln(2)\frac{\varepsilon}{2}$. Choose $\alpha$ small enough
such that $\eta_\alpha(|f'(x_l)|)=\ln(|f'(x_l)|)$ holds for
all $l\leq N$. Then, for sufficiently high $p$,
$\frac{1}{N}\sum^{N-1}_{l=K(N)}\ln(|f'(x_l)|)>
\ln(2)\frac{\varepsilon}{2}$ follows. In the second statement,
the sum can be split the following way:
$\frac{1}{N}\sum^{K(N)-1}_{l=0}\ln(|f'(x_l)|)+
\frac{1}{N}\sum^{N-1}_{l=K(N)}\ln(|f'(x_l)|) < \lambda(x)+
\ln(2)\frac{\varepsilon}{4}$. The first addend on the left side
is bounded form below by
$\frac{K(N)}{N}(\lambda(x)-\ln(2)\frac{\varepsilon}{4})\geq
\lambda(x)-\ln(2)\frac{\varepsilon}{4}$, the second addend is
bounded from below by $\ln(2)\frac{\varepsilon}{2}$. Hence, 
$\lambda(x)+\ln(2)\frac{\varepsilon}{4}<\frac{1}{N}\sum_{l=0}^{N-1}
\ln(|f'(x_l)|)<\lambda(x)+\ln(2)\frac{\varepsilon}{4}$
follows, but this is a contradiction.
\end{proof}

In the end, it is shown that, if $\lambda(x)=\overline{\lambda}(x)$
holds, the algorithm presented here
is optimal with respect to the loss of significance rate. This
means that no algorithm with the specification presented at the
beginning of this section has a lower loss of significance
rate than the algorithm presented in this section.
\begin{proposition}
Let $(x_n)_n$ be an orbit of the dynamical system $(D,f)$ and
$\lambda(x_0)>0$. Then, for any $\varepsilon>0$ there exists
an $N_0\in\IN$ such that for any $N\geq N_0$ there is some
$\delta>0$ such that the following holds. Let an initial
value $y_0\in[x_0-\delta,x_0+\delta]\cap D$ be given
and consider the corresponding orbit $(y_n)_n$. Then,
\begin{equation*}
|y_N - x_N|\geq e^{N(\lambda(x_0) - \varepsilon)}|y_0 - x_0|
\end{equation*}
holds.
\end{proposition}
\begin{proof}
Let $\varepsilon'>0$ be given. Then there exists some $N_0\in\IN$
such that $\frac{1}{N}\sum_{n=0}^{N-1}\ln(|f'(x_n)|)\geq
\lambda(x_0) - \varepsilon'$ holds for all $N\geq N_0$. For given
$N\geq N_0$ there is some $\delta'>0$ with
$\min\{|f'(x_n)| : 0\leq n\leq N\}>\delta'$, otherwise
$\lambda(x_0)$ would not exist. Consider now an
orbit $(y_n)_n$ with $|y_n - x_n| \leq 2\sigma\frac{\delta'}{L'}$
for $n=0,\dots,N$ where $L'$ is a Lipschitz constant of $f'$
and $0<\sigma<1$ arbitrary. Then, for $n=1,\dots,N$, the following
estimation holds.
\begin{align*}
|y_n - x_n| &= |f(y_{n-1}) - f(x_{n-1})|\\
 &= |f'(x_{n-1})(y_{n-1} - x_{n-1}) + \frac{1}{2}f''(\xi_{n-1})
    (y_{n-1} - x_{n-1})^2|\\
 &\geq (|f'(x_{n-1})| - \frac{1}{2}|f''(\xi_{n-1})|\cdot |y_{n-1} - x_{n-1}|)
       |y_{n-1} - x_{n-1}|\\
 &\geq (|f'(x_{n-1})| - \frac{1}{2}L'|y_{n-1} - x_{n-1}|)|y_{n-1} - x_{n-1}|\\
 &\geq (|f'(x_{n-1})| - \sigma\delta')|y_{n-1} - x_{n-1}|
\end{align*}
where $\xi_{n-1}\in[x_{n-1},y_{n-1}]\cup[y_{n-1},x_{n-1}]$.
Iterating finally gives
$|y_N - x_N|\geq\prod_{n=0}^{N-1}(|f'(x_n)| - \sigma\delta')|y_0 - x_0|$.
Now determine some constant $C>0$ such that
$\ln(|f'(x_n)|-\sigma\delta')\geq \ln(|f'(x_n)|) - C$ holds for all
$n=0,\dots,N$ the following way. A short calculation shows that this
is equivalent to $C\geq\ln(1+\frac{\sigma\delta'}{|f'(x_n)|-\sigma\delta'})$.
Using $\ln(1+\frac{\sigma\delta'}{|f'(x_n)|-\sigma\delta'})\leq
\ln(1+\frac{\sigma\delta'}{\delta'-\sigma\delta'})=
\ln(1+\frac{\sigma}{1-\sigma})\leq\frac{\sigma}{1-\sigma}$
finally gives $C\geq\frac{\sigma}{1-\sigma}$ as a sufficient condition.
Set $C=\frac{\sigma}{1-\sigma}$. Let $\varepsilon > 0$ be given. Set
$C=\varepsilon'=\varepsilon/2$, then
\begin{align*}
\sum_{n=0}^{N-1}\ln(|f'(x_n)| -\frac{C\delta'}{1+C})
 &\geq \sum_{n=0}^{N-1}\ln(|f'(x_n)|) - NC\\
 &\geq N(\lambda(x_0) -\varepsilon' - C)
\end{align*}
and hence
\begin{equation*}
|y_N - x_N|\geq e^{N(\lambda(x_0) - \varepsilon' - C)}|y_0 - x_0|
\end{equation*}
follows.
\end{proof}
\begin{proposition}
Let $(D,f)$ be as in Assumption \ref{ass:main} and
$x\in\IQ\cap D$ given such that $\lambda(x)$ exists and
$\lambda(x)>0$. Consider an algorithm computing an initial
segment $(x_n)_{0\leq n\leq N}$ of the orbit of $x$ with relative
error $\leq 10^{-p}$ for some $N\in\IN$, $p\in\IZ$.
Then the algorithm has a loss of significance
rate $\sigma(x)\geq\frac{1}{\ln(2)}\lambda(x)$.
\end{proposition}
\begin{proof}
Let $0<\varepsilon<\lambda(x)$ be given and $N$ big enough
such that the previous proposition holds for some $\delta>0$.
Choose some $y\in[x-\delta,x+\delta]\cap D$,
$y\neq x$. Let $p_0$ be big enough such that
\begin{equation}
\label{cond}
2M10^{-p_0} < e^{N(\lambda(x)-\varepsilon)}|y - x|
\end{equation}
holds, where $M>0$ such that $|x|\leq M$ for all $x\in D$.
Consider $y$ as the initial value of another orbit $(y_n)_n$.
Then, with the above proposition,
\begin{equation}
\label{sep}
|y_N - x_N|>10^{-p_0}(|x_N|+|y_N|)
\end{equation}
follows. Condition (\ref{cond}) can also be written as
$\delta>2M10^{-p_0}e^{-N(\lambda(x)-\varepsilon)}$.

Consider now some precision $m$, the algorithm actually
is working with. Assume for simplicity further that for the
initial value $x=\hat{x}_0$ holds and assume without loss of
generality $\delta<\frac{1}{2}ulp(\hat{x}_0)$. Then, first,
$\hat{y}_0=\hat{x}_0$ holds. Second, the above condition gives 
$2^{-(m+1)}>10^{-p_0}e^{-N(\lambda(x)-\varepsilon)}$ since
$ulp(\hat{x}_0)\leq 2^{-m+1}|\hat{x}_0|\leq 2^{-m+1}M$ holds.
In other words, the above condition gives an upper bound
$m<p_0\cdot\ld(10)+N(\lambda(x)-\varepsilon)/\ln(2) - 1$ on
the needed precision $m$. Assume furthermore that $m$ is
big enough such that $\hat{x}_N$ and $\hat{y}_N$ is computed
with the demanded precision, that is
$|\hat{x}_N - x_N|\leq 10^{-p_0}|x_N|$ and
$|\hat{y}_N - y_N|\leq 10^{-p_0}|y_N|$ holds. Using
(\ref{sep}) gives $\hat{x}_N\neq\hat{y}_N$.
But this is a contradiction since $\hat{y}_0=\hat{x}_0$.
So the upper bound on $m$ calculated above is still too small.
Hence, $m\geq p_0\cdot\ld(10)+N(\lambda(x)-\varepsilon)/\ln(2)-1$
must hold. Since Condition (\ref{cond}) also holds for any $N'>N$
and the same $p_0$ as well as for any $p\geq p_0$,
$\sigma(x,p)\geq(\lambda(x) - \varepsilon)/\ln(2)$
follows for all $p\geq p_0$. Computing $\sigma(x)$ finally
gives the assertion.
\end{proof}
Furthermore, if $\overline{\lambda}(x)=\lambda(x)$ holds, then Corollary
\ref{main_stat1} gives $\sigma(x)=\frac{1}{\ln(2)}\max(0,\lambda(x))$
for the algorithm presented at the beginning of this section.
Using the above proposition then leads to the following theorem.
\begin{theorem}
Let $(D,f)$ be as in Assumption \ref{ass:main} and
$x\in\IQ\cap D$ given such that $\lambda(x)$
exists and $\overline{\lambda}(x)=\lambda(x)$ holds. Consider an
algorithm computing an initial segment $(x_n)_{0\leq n\leq N}$
of the orbit of $x$ with relative error $\leq 10^{-p}$. Then this
algorithm has a loss of significance rate greater or equal to that
of the algorithm specified by the recursion (\ref{comp_rec1}) and
(\ref{comp_rec2}).
\end{theorem}

\section{Conclusions}
In this paper, two main issues are addressed. First it is
shown that a mathematically rigorous treatment of the
computability aspects of the iteration of a real function
in terms of arbitrary-precision floating-point arithmetic
including automated error analysis is straightforward.
Also, this treatment is in a manner which is familiar to people
working in the field of numerical analysis or scientific computing
and also for theoretical computer scientists. Furthermore, the
approach does not only allow answers concerning the existence of
an algorithm which meets the requirements of computability theory,
but it also allows a treatment of its space complexity in form of
the loss of significance rate (which is actually the lookahead in
Computable Analysis) and optimality as discussed in the preceding
section. As a consequence, the approach here enables a motivated
reader the real implementation and supports a practical performance
analysis.

Second, the results show that the Lyapunov exponent, a central
quantity in dynamical systems theory, also finds its way into
complexity theory, a branch in theoretical computer science.
In dynamical systems theory, the Lyapunov exponent describes
the rate of divergence in the course of time of initially
infinitesimal nearby states. For two states having a small but
finite initial separation, the Lyapunov exponent has only
relevance for short time scales \cite{ce06}. The reason is that
due to the boundedness of the phase space, any two different
orbits cannot separate arbitrarily far away. However, the loss
of significance rate shows that the Lyapunov exponent has on
long time scales not only an asymptotic significance but also
a concrete practical one.

\subsection*{Acknowledgments}

The author wishes to express his gratitude to Peter Hertling
for helpful discussions and comments.

\bibliographystyle{model1b-num-names}

\input{CompLiap.bbl}
\end{document}

%% file: figure3.tex
\setlength{\unitlength}{0.240900pt}
\ifx\plotpoint\undefined\newsavebox{\plotpoint}\fi
\begin{picture}(1500,900)(0,0)
\sbox{\plotpoint}{\rule[-0.200pt]{0.400pt}{0.400pt}}%
\put(171.0,131.0){\rule[-0.200pt]{4.818pt}{0.400pt}}
\put(151,131){\makebox(0,0)[r]{$0.0$}}
\put(1429.0,131.0){\rule[-0.200pt]{4.818pt}{0.400pt}}
\put(171.0,222.0){\rule[-0.200pt]{4.818pt}{0.400pt}}
\put(151,222){\makebox(0,0)[r]{$1.0$}}
\put(1429.0,222.0){\rule[-0.200pt]{4.818pt}{0.400pt}}
\put(171.0,313.0){\rule[-0.200pt]{4.818pt}{0.400pt}}
\put(151,313){\makebox(0,0)[r]{$2.0$}}
\put(1429.0,313.0){\rule[-0.200pt]{4.818pt}{0.400pt}}
\put(171.0,404.0){\rule[-0.200pt]{4.818pt}{0.400pt}}
\put(151,404){\makebox(0,0)[r]{$3.0$}}
\put(1429.0,404.0){\rule[-0.200pt]{4.818pt}{0.400pt}}
\put(171.0,495.0){\rule[-0.200pt]{4.818pt}{0.400pt}}
\put(151,495){\makebox(0,0)[r]{$4.0$}}
\put(1429.0,495.0){\rule[-0.200pt]{4.818pt}{0.400pt}}
\put(171.0,586.0){\rule[-0.200pt]{4.818pt}{0.400pt}}
\put(151,586){\makebox(0,0)[r]{$5.0$}}
\put(1429.0,586.0){\rule[-0.200pt]{4.818pt}{0.400pt}}
\put(171.0,677.0){\rule[-0.200pt]{4.818pt}{0.400pt}}
\put(151,677){\makebox(0,0)[r]{$6.0$}}
\put(1429.0,677.0){\rule[-0.200pt]{4.818pt}{0.400pt}}
\put(171.0,768.0){\rule[-0.200pt]{4.818pt}{0.400pt}}
\put(151,768){\makebox(0,0)[r]{$7.0$}}
\put(1429.0,768.0){\rule[-0.200pt]{4.818pt}{0.400pt}}
\put(171.0,859.0){\rule[-0.200pt]{4.818pt}{0.400pt}}
\put(151,859){\makebox(0,0)[r]{$8.0$}}
\put(1429.0,859.0){\rule[-0.200pt]{4.818pt}{0.400pt}}
\put(171.0,131.0){\rule[-0.200pt]{0.400pt}{4.818pt}}
\put(171,90){\makebox(0,0){$0.0$}}
\put(171.0,839.0){\rule[-0.200pt]{0.400pt}{4.818pt}}
\put(331.0,131.0){\rule[-0.200pt]{0.400pt}{4.818pt}}
\put(331,90){\makebox(0,0){$0.5$}}
\put(331.0,839.0){\rule[-0.200pt]{0.400pt}{4.818pt}}
\put(491.0,131.0){\rule[-0.200pt]{0.400pt}{4.818pt}}
\put(491,90){\makebox(0,0){$1.0$}}
\put(491.0,839.0){\rule[-0.200pt]{0.400pt}{4.818pt}}
\put(650.0,131.0){\rule[-0.200pt]{0.400pt}{4.818pt}}
\put(650,90){\makebox(0,0){$1.5$}}
\put(650.0,839.0){\rule[-0.200pt]{0.400pt}{4.818pt}}
\put(810.0,131.0){\rule[-0.200pt]{0.400pt}{4.818pt}}
\put(810,90){\makebox(0,0){$2.0$}}
\put(810.0,839.0){\rule[-0.200pt]{0.400pt}{4.818pt}}
\put(970.0,131.0){\rule[-0.200pt]{0.400pt}{4.818pt}}
\put(970,90){\makebox(0,0){$2.5$}}
\put(970.0,839.0){\rule[-0.200pt]{0.400pt}{4.818pt}}
\put(1130.0,131.0){\rule[-0.200pt]{0.400pt}{4.818pt}}
\put(1130,90){\makebox(0,0){$3.0$}}
\put(1130.0,839.0){\rule[-0.200pt]{0.400pt}{4.818pt}}
\put(1289.0,131.0){\rule[-0.200pt]{0.400pt}{4.818pt}}
\put(1289,90){\makebox(0,0){$3.5$}}
\put(1289.0,839.0){\rule[-0.200pt]{0.400pt}{4.818pt}}
\put(1449.0,131.0){\rule[-0.200pt]{0.400pt}{4.818pt}}
\put(1449,90){\makebox(0,0){$4.0$}}
\put(1449.0,839.0){\rule[-0.200pt]{0.400pt}{4.818pt}}
\put(171.0,131.0){\rule[-0.200pt]{0.400pt}{175.375pt}}
\put(171.0,131.0){\rule[-0.200pt]{307.870pt}{0.400pt}}
\put(1449.0,131.0){\rule[-0.200pt]{0.400pt}{175.375pt}}
\put(171.0,859.0){\rule[-0.200pt]{307.870pt}{0.400pt}}
\put(50,495){\makebox(0,0){\rotatebox{90}{$\sigma_{est}$}}}
\put(810,29){\makebox(0,0){$\mu$}}
\put(173,827){\usebox{\plotpoint}}
\put(172.67,736){\rule{0.400pt}{21.922pt}}
\multiput(172.17,781.50)(1.000,-45.500){2}{\rule{0.400pt}{10.961pt}}
\put(174.17,683){\rule{0.400pt}{10.700pt}}
\multiput(173.17,713.79)(2.000,-30.792){2}{\rule{0.400pt}{5.350pt}}
\put(175.67,645){\rule{0.400pt}{9.154pt}}
\multiput(175.17,664.00)(1.000,-19.000){2}{\rule{0.400pt}{4.577pt}}
\put(177.17,616){\rule{0.400pt}{5.900pt}}
\multiput(176.17,632.75)(2.000,-16.754){2}{\rule{0.400pt}{2.950pt}}
\put(179.17,592){\rule{0.400pt}{4.900pt}}
\multiput(178.17,605.83)(2.000,-13.830){2}{\rule{0.400pt}{2.450pt}}
\put(180.67,572){\rule{0.400pt}{4.818pt}}
\multiput(180.17,582.00)(1.000,-10.000){2}{\rule{0.400pt}{2.409pt}}
\put(182.17,554){\rule{0.400pt}{3.700pt}}
\multiput(181.17,564.32)(2.000,-10.320){2}{\rule{0.400pt}{1.850pt}}
\put(183.67,539){\rule{0.400pt}{3.614pt}}
\multiput(183.17,546.50)(1.000,-7.500){2}{\rule{0.400pt}{1.807pt}}
\put(185.17,525){\rule{0.400pt}{2.900pt}}
\multiput(184.17,532.98)(2.000,-7.981){2}{\rule{0.400pt}{1.450pt}}
\put(187.17,513){\rule{0.400pt}{2.500pt}}
\multiput(186.17,519.81)(2.000,-6.811){2}{\rule{0.400pt}{1.250pt}}
\put(188.67,501){\rule{0.400pt}{2.891pt}}
\multiput(188.17,507.00)(1.000,-6.000){2}{\rule{0.400pt}{1.445pt}}
\put(190.17,491){\rule{0.400pt}{2.100pt}}
\multiput(189.17,496.64)(2.000,-5.641){2}{\rule{0.400pt}{1.050pt}}
\put(191.67,481){\rule{0.400pt}{2.409pt}}
\multiput(191.17,486.00)(1.000,-5.000){2}{\rule{0.400pt}{1.204pt}}
\put(193.17,472){\rule{0.400pt}{1.900pt}}
\multiput(192.17,477.06)(2.000,-5.056){2}{\rule{0.400pt}{0.950pt}}
\put(195.17,463){\rule{0.400pt}{1.900pt}}
\multiput(194.17,468.06)(2.000,-5.056){2}{\rule{0.400pt}{0.950pt}}
\put(196.67,456){\rule{0.400pt}{1.686pt}}
\multiput(196.17,459.50)(1.000,-3.500){2}{\rule{0.400pt}{0.843pt}}
\put(198.17,448){\rule{0.400pt}{1.700pt}}
\multiput(197.17,452.47)(2.000,-4.472){2}{\rule{0.400pt}{0.850pt}}
\put(199.67,441){\rule{0.400pt}{1.686pt}}
\multiput(199.17,444.50)(1.000,-3.500){2}{\rule{0.400pt}{0.843pt}}
\put(201.17,434){\rule{0.400pt}{1.500pt}}
\multiput(200.17,437.89)(2.000,-3.887){2}{\rule{0.400pt}{0.750pt}}
\put(203.17,428){\rule{0.400pt}{1.300pt}}
\multiput(202.17,431.30)(2.000,-3.302){2}{\rule{0.400pt}{0.650pt}}
\put(204.67,422){\rule{0.400pt}{1.445pt}}
\multiput(204.17,425.00)(1.000,-3.000){2}{\rule{0.400pt}{0.723pt}}
\put(206.17,416){\rule{0.400pt}{1.300pt}}
\multiput(205.17,419.30)(2.000,-3.302){2}{\rule{0.400pt}{0.650pt}}
\put(207.67,410){\rule{0.400pt}{1.445pt}}
\multiput(207.17,413.00)(1.000,-3.000){2}{\rule{0.400pt}{0.723pt}}
\put(209.17,405){\rule{0.400pt}{1.100pt}}
\multiput(208.17,407.72)(2.000,-2.717){2}{\rule{0.400pt}{0.550pt}}
\put(211.17,400){\rule{0.400pt}{1.100pt}}
\multiput(210.17,402.72)(2.000,-2.717){2}{\rule{0.400pt}{0.550pt}}
\put(212.67,395){\rule{0.400pt}{1.204pt}}
\multiput(212.17,397.50)(1.000,-2.500){2}{\rule{0.400pt}{0.602pt}}
\put(214.17,390){\rule{0.400pt}{1.100pt}}
\multiput(213.17,392.72)(2.000,-2.717){2}{\rule{0.400pt}{0.550pt}}
\put(215.67,385){\rule{0.400pt}{1.204pt}}
\multiput(215.17,387.50)(1.000,-2.500){2}{\rule{0.400pt}{0.602pt}}
\put(217.17,381){\rule{0.400pt}{0.900pt}}
\multiput(216.17,383.13)(2.000,-2.132){2}{\rule{0.400pt}{0.450pt}}
\put(219.17,377){\rule{0.400pt}{0.900pt}}
\multiput(218.17,379.13)(2.000,-2.132){2}{\rule{0.400pt}{0.450pt}}
\put(220.67,373){\rule{0.400pt}{0.964pt}}
\multiput(220.17,375.00)(1.000,-2.000){2}{\rule{0.400pt}{0.482pt}}
\put(222.17,369){\rule{0.400pt}{0.900pt}}
\multiput(221.17,371.13)(2.000,-2.132){2}{\rule{0.400pt}{0.450pt}}
\put(223.67,365){\rule{0.400pt}{0.964pt}}
\multiput(223.17,367.00)(1.000,-2.000){2}{\rule{0.400pt}{0.482pt}}
\put(225.17,361){\rule{0.400pt}{0.900pt}}
\multiput(224.17,363.13)(2.000,-2.132){2}{\rule{0.400pt}{0.450pt}}
\put(227.17,357){\rule{0.400pt}{0.900pt}}
\multiput(226.17,359.13)(2.000,-2.132){2}{\rule{0.400pt}{0.450pt}}
\put(228.67,353){\rule{0.400pt}{0.964pt}}
\multiput(228.17,355.00)(1.000,-2.000){2}{\rule{0.400pt}{0.482pt}}
\put(230.17,350){\rule{0.400pt}{0.700pt}}
\multiput(229.17,351.55)(2.000,-1.547){2}{\rule{0.400pt}{0.350pt}}
\put(231.67,347){\rule{0.400pt}{0.723pt}}
\multiput(231.17,348.50)(1.000,-1.500){2}{\rule{0.400pt}{0.361pt}}
\put(233.17,343){\rule{0.400pt}{0.900pt}}
\multiput(232.17,345.13)(2.000,-2.132){2}{\rule{0.400pt}{0.450pt}}
\put(234.67,340){\rule{0.400pt}{0.723pt}}
\multiput(234.17,341.50)(1.000,-1.500){2}{\rule{0.400pt}{0.361pt}}
\put(236.17,337){\rule{0.400pt}{0.700pt}}
\multiput(235.17,338.55)(2.000,-1.547){2}{\rule{0.400pt}{0.350pt}}
\put(238.17,334){\rule{0.400pt}{0.700pt}}
\multiput(237.17,335.55)(2.000,-1.547){2}{\rule{0.400pt}{0.350pt}}
\put(239.67,331){\rule{0.400pt}{0.723pt}}
\multiput(239.17,332.50)(1.000,-1.500){2}{\rule{0.400pt}{0.361pt}}
\put(241.17,328){\rule{0.400pt}{0.700pt}}
\multiput(240.17,329.55)(2.000,-1.547){2}{\rule{0.400pt}{0.350pt}}
\put(242.67,325){\rule{0.400pt}{0.723pt}}
\multiput(242.17,326.50)(1.000,-1.500){2}{\rule{0.400pt}{0.361pt}}
\put(244.17,322){\rule{0.400pt}{0.700pt}}
\multiput(243.17,323.55)(2.000,-1.547){2}{\rule{0.400pt}{0.350pt}}
\put(246.17,319){\rule{0.400pt}{0.700pt}}
\multiput(245.17,320.55)(2.000,-1.547){2}{\rule{0.400pt}{0.350pt}}
\put(247.67,317){\rule{0.400pt}{0.482pt}}
\multiput(247.17,318.00)(1.000,-1.000){2}{\rule{0.400pt}{0.241pt}}
\put(249.17,314){\rule{0.400pt}{0.700pt}}
\multiput(248.17,315.55)(2.000,-1.547){2}{\rule{0.400pt}{0.350pt}}
\put(250.67,311){\rule{0.400pt}{0.723pt}}
\multiput(250.17,312.50)(1.000,-1.500){2}{\rule{0.400pt}{0.361pt}}
\put(252,309.17){\rule{0.482pt}{0.400pt}}
\multiput(252.00,310.17)(1.000,-2.000){2}{\rule{0.241pt}{0.400pt}}
\put(254.17,306){\rule{0.400pt}{0.700pt}}
\multiput(253.17,307.55)(2.000,-1.547){2}{\rule{0.400pt}{0.350pt}}
\put(255.67,304){\rule{0.400pt}{0.482pt}}
\multiput(255.17,305.00)(1.000,-1.000){2}{\rule{0.400pt}{0.241pt}}
\put(257.17,301){\rule{0.400pt}{0.700pt}}
\multiput(256.17,302.55)(2.000,-1.547){2}{\rule{0.400pt}{0.350pt}}
\put(258.67,299){\rule{0.400pt}{0.482pt}}
\multiput(258.17,300.00)(1.000,-1.000){2}{\rule{0.400pt}{0.241pt}}
\put(260,297.17){\rule{0.482pt}{0.400pt}}
\multiput(260.00,298.17)(1.000,-2.000){2}{\rule{0.241pt}{0.400pt}}
\put(262.17,294){\rule{0.400pt}{0.700pt}}
\multiput(261.17,295.55)(2.000,-1.547){2}{\rule{0.400pt}{0.350pt}}
\put(263.67,292){\rule{0.400pt}{0.482pt}}
\multiput(263.17,293.00)(1.000,-1.000){2}{\rule{0.400pt}{0.241pt}}
\put(265,290.17){\rule{0.482pt}{0.400pt}}
\multiput(265.00,291.17)(1.000,-2.000){2}{\rule{0.241pt}{0.400pt}}
\put(266.67,288){\rule{0.400pt}{0.482pt}}
\multiput(266.17,289.00)(1.000,-1.000){2}{\rule{0.400pt}{0.241pt}}
\put(268,286.17){\rule{0.482pt}{0.400pt}}
\multiput(268.00,287.17)(1.000,-2.000){2}{\rule{0.241pt}{0.400pt}}
\put(270,284.17){\rule{0.482pt}{0.400pt}}
\multiput(270.00,285.17)(1.000,-2.000){2}{\rule{0.241pt}{0.400pt}}
\put(271.67,282){\rule{0.400pt}{0.482pt}}
\multiput(271.17,283.00)(1.000,-1.000){2}{\rule{0.400pt}{0.241pt}}
\put(273,280.17){\rule{0.482pt}{0.400pt}}
\multiput(273.00,281.17)(1.000,-2.000){2}{\rule{0.241pt}{0.400pt}}
\put(274.67,278){\rule{0.400pt}{0.482pt}}
\multiput(274.17,279.00)(1.000,-1.000){2}{\rule{0.400pt}{0.241pt}}
\put(276,276.17){\rule{0.482pt}{0.400pt}}
\multiput(276.00,277.17)(1.000,-2.000){2}{\rule{0.241pt}{0.400pt}}
\put(278,274.17){\rule{0.482pt}{0.400pt}}
\multiput(278.00,275.17)(1.000,-2.000){2}{\rule{0.241pt}{0.400pt}}
\put(279.67,272){\rule{0.400pt}{0.482pt}}
\multiput(279.17,273.00)(1.000,-1.000){2}{\rule{0.400pt}{0.241pt}}
\put(281,270.17){\rule{0.482pt}{0.400pt}}
\multiput(281.00,271.17)(1.000,-2.000){2}{\rule{0.241pt}{0.400pt}}
\put(282.67,268){\rule{0.400pt}{0.482pt}}
\multiput(282.17,269.00)(1.000,-1.000){2}{\rule{0.400pt}{0.241pt}}
\put(284,266.17){\rule{0.482pt}{0.400pt}}
\multiput(284.00,267.17)(1.000,-2.000){2}{\rule{0.241pt}{0.400pt}}
\put(286,264.17){\rule{0.482pt}{0.400pt}}
\multiput(286.00,265.17)(1.000,-2.000){2}{\rule{0.241pt}{0.400pt}}
\put(288,262.67){\rule{0.241pt}{0.400pt}}
\multiput(288.00,263.17)(0.500,-1.000){2}{\rule{0.120pt}{0.400pt}}
\put(289,261.17){\rule{0.482pt}{0.400pt}}
\multiput(289.00,262.17)(1.000,-2.000){2}{\rule{0.241pt}{0.400pt}}
\put(290.67,259){\rule{0.400pt}{0.482pt}}
\multiput(290.17,260.00)(1.000,-1.000){2}{\rule{0.400pt}{0.241pt}}
\put(292,257.17){\rule{0.482pt}{0.400pt}}
\multiput(292.00,258.17)(1.000,-2.000){2}{\rule{0.241pt}{0.400pt}}
\put(294,255.67){\rule{0.482pt}{0.400pt}}
\multiput(294.00,256.17)(1.000,-1.000){2}{\rule{0.241pt}{0.400pt}}
\put(295.67,254){\rule{0.400pt}{0.482pt}}
\multiput(295.17,255.00)(1.000,-1.000){2}{\rule{0.400pt}{0.241pt}}
\put(297,252.17){\rule{0.482pt}{0.400pt}}
\multiput(297.00,253.17)(1.000,-2.000){2}{\rule{0.241pt}{0.400pt}}
\put(299,250.67){\rule{0.241pt}{0.400pt}}
\multiput(299.00,251.17)(0.500,-1.000){2}{\rule{0.120pt}{0.400pt}}
\put(300,249.17){\rule{0.482pt}{0.400pt}}
\multiput(300.00,250.17)(1.000,-2.000){2}{\rule{0.241pt}{0.400pt}}
\put(302,247.17){\rule{0.482pt}{0.400pt}}
\multiput(302.00,248.17)(1.000,-2.000){2}{\rule{0.241pt}{0.400pt}}
\put(304,245.67){\rule{0.241pt}{0.400pt}}
\multiput(304.00,246.17)(0.500,-1.000){2}{\rule{0.120pt}{0.400pt}}
\put(305,244.17){\rule{0.482pt}{0.400pt}}
\multiput(305.00,245.17)(1.000,-2.000){2}{\rule{0.241pt}{0.400pt}}
\put(307,242.67){\rule{0.241pt}{0.400pt}}
\multiput(307.00,243.17)(0.500,-1.000){2}{\rule{0.120pt}{0.400pt}}
\put(308,241.17){\rule{0.482pt}{0.400pt}}
\multiput(308.00,242.17)(1.000,-2.000){2}{\rule{0.241pt}{0.400pt}}
\put(310,239.67){\rule{0.482pt}{0.400pt}}
\multiput(310.00,240.17)(1.000,-1.000){2}{\rule{0.241pt}{0.400pt}}
\put(311.67,238){\rule{0.400pt}{0.482pt}}
\multiput(311.17,239.00)(1.000,-1.000){2}{\rule{0.400pt}{0.241pt}}
\put(313,236.67){\rule{0.482pt}{0.400pt}}
\multiput(313.00,237.17)(1.000,-1.000){2}{\rule{0.241pt}{0.400pt}}
\put(314.67,235){\rule{0.400pt}{0.482pt}}
\multiput(314.17,236.00)(1.000,-1.000){2}{\rule{0.400pt}{0.241pt}}
\put(316,233.67){\rule{0.482pt}{0.400pt}}
\multiput(316.00,234.17)(1.000,-1.000){2}{\rule{0.241pt}{0.400pt}}
\put(318,232.67){\rule{0.482pt}{0.400pt}}
\multiput(318.00,233.17)(1.000,-1.000){2}{\rule{0.241pt}{0.400pt}}
\put(319.67,231){\rule{0.400pt}{0.482pt}}
\multiput(319.17,232.00)(1.000,-1.000){2}{\rule{0.400pt}{0.241pt}}
\put(321,229.67){\rule{0.482pt}{0.400pt}}
\multiput(321.00,230.17)(1.000,-1.000){2}{\rule{0.241pt}{0.400pt}}
\put(322.67,228){\rule{0.400pt}{0.482pt}}
\multiput(322.17,229.00)(1.000,-1.000){2}{\rule{0.400pt}{0.241pt}}
\put(324,226.67){\rule{0.482pt}{0.400pt}}
\multiput(324.00,227.17)(1.000,-1.000){2}{\rule{0.241pt}{0.400pt}}
\put(326,225.67){\rule{0.482pt}{0.400pt}}
\multiput(326.00,226.17)(1.000,-1.000){2}{\rule{0.241pt}{0.400pt}}
\put(327.67,224){\rule{0.400pt}{0.482pt}}
\multiput(327.17,225.00)(1.000,-1.000){2}{\rule{0.400pt}{0.241pt}}
\put(329,222.67){\rule{0.482pt}{0.400pt}}
\multiput(329.00,223.17)(1.000,-1.000){2}{\rule{0.241pt}{0.400pt}}
\put(331,221.67){\rule{0.241pt}{0.400pt}}
\multiput(331.00,222.17)(0.500,-1.000){2}{\rule{0.120pt}{0.400pt}}
\put(332,220.17){\rule{0.482pt}{0.400pt}}
\multiput(332.00,221.17)(1.000,-2.000){2}{\rule{0.241pt}{0.400pt}}
\put(334,218.67){\rule{0.482pt}{0.400pt}}
\multiput(334.00,219.17)(1.000,-1.000){2}{\rule{0.241pt}{0.400pt}}
\put(336,217.67){\rule{0.241pt}{0.400pt}}
\multiput(336.00,218.17)(0.500,-1.000){2}{\rule{0.120pt}{0.400pt}}
\put(337,216.67){\rule{0.482pt}{0.400pt}}
\multiput(337.00,217.17)(1.000,-1.000){2}{\rule{0.241pt}{0.400pt}}
\put(338.67,215){\rule{0.400pt}{0.482pt}}
\multiput(338.17,216.00)(1.000,-1.000){2}{\rule{0.400pt}{0.241pt}}
\put(340,213.67){\rule{0.482pt}{0.400pt}}
\multiput(340.00,214.17)(1.000,-1.000){2}{\rule{0.241pt}{0.400pt}}
\put(342,212.67){\rule{0.482pt}{0.400pt}}
\multiput(342.00,213.17)(1.000,-1.000){2}{\rule{0.241pt}{0.400pt}}
\put(344,211.67){\rule{0.241pt}{0.400pt}}
\multiput(344.00,212.17)(0.500,-1.000){2}{\rule{0.120pt}{0.400pt}}
\put(345,210.67){\rule{0.482pt}{0.400pt}}
\multiput(345.00,211.17)(1.000,-1.000){2}{\rule{0.241pt}{0.400pt}}
\put(346.67,209){\rule{0.400pt}{0.482pt}}
\multiput(346.17,210.00)(1.000,-1.000){2}{\rule{0.400pt}{0.241pt}}
\put(348,207.67){\rule{0.482pt}{0.400pt}}
\multiput(348.00,208.17)(1.000,-1.000){2}{\rule{0.241pt}{0.400pt}}
\put(350,206.67){\rule{0.482pt}{0.400pt}}
\multiput(350.00,207.17)(1.000,-1.000){2}{\rule{0.241pt}{0.400pt}}
\put(352,205.67){\rule{0.241pt}{0.400pt}}
\multiput(352.00,206.17)(0.500,-1.000){2}{\rule{0.120pt}{0.400pt}}
\put(353,204.67){\rule{0.482pt}{0.400pt}}
\multiput(353.00,205.17)(1.000,-1.000){2}{\rule{0.241pt}{0.400pt}}
\put(355,203.67){\rule{0.241pt}{0.400pt}}
\multiput(355.00,204.17)(0.500,-1.000){2}{\rule{0.120pt}{0.400pt}}
\put(356,202.17){\rule{0.482pt}{0.400pt}}
\multiput(356.00,203.17)(1.000,-2.000){2}{\rule{0.241pt}{0.400pt}}
\put(358,200.67){\rule{0.482pt}{0.400pt}}
\multiput(358.00,201.17)(1.000,-1.000){2}{\rule{0.241pt}{0.400pt}}
\put(360,199.67){\rule{0.241pt}{0.400pt}}
\multiput(360.00,200.17)(0.500,-1.000){2}{\rule{0.120pt}{0.400pt}}
\put(361,198.67){\rule{0.482pt}{0.400pt}}
\multiput(361.00,199.17)(1.000,-1.000){2}{\rule{0.241pt}{0.400pt}}
\put(363,197.67){\rule{0.241pt}{0.400pt}}
\multiput(363.00,198.17)(0.500,-1.000){2}{\rule{0.120pt}{0.400pt}}
\put(364,196.67){\rule{0.482pt}{0.400pt}}
\multiput(364.00,197.17)(1.000,-1.000){2}{\rule{0.241pt}{0.400pt}}
\put(366,195.67){\rule{0.241pt}{0.400pt}}
\multiput(366.00,196.17)(0.500,-1.000){2}{\rule{0.120pt}{0.400pt}}
\put(367,194.67){\rule{0.482pt}{0.400pt}}
\multiput(367.00,195.17)(1.000,-1.000){2}{\rule{0.241pt}{0.400pt}}
\put(369,193.67){\rule{0.482pt}{0.400pt}}
\multiput(369.00,194.17)(1.000,-1.000){2}{\rule{0.241pt}{0.400pt}}
\put(371,192.67){\rule{0.241pt}{0.400pt}}
\multiput(371.00,193.17)(0.500,-1.000){2}{\rule{0.120pt}{0.400pt}}
\put(372,191.67){\rule{0.482pt}{0.400pt}}
\multiput(372.00,192.17)(1.000,-1.000){2}{\rule{0.241pt}{0.400pt}}
\put(374,190.67){\rule{0.241pt}{0.400pt}}
\multiput(374.00,191.17)(0.500,-1.000){2}{\rule{0.120pt}{0.400pt}}
\put(375,189.67){\rule{0.482pt}{0.400pt}}
\multiput(375.00,190.17)(1.000,-1.000){2}{\rule{0.241pt}{0.400pt}}
\put(377,188.67){\rule{0.482pt}{0.400pt}}
\multiput(377.00,189.17)(1.000,-1.000){2}{\rule{0.241pt}{0.400pt}}
\put(379,187.67){\rule{0.241pt}{0.400pt}}
\multiput(379.00,188.17)(0.500,-1.000){2}{\rule{0.120pt}{0.400pt}}
\put(380,186.67){\rule{0.482pt}{0.400pt}}
\multiput(380.00,187.17)(1.000,-1.000){2}{\rule{0.241pt}{0.400pt}}
\put(382,185.67){\rule{0.241pt}{0.400pt}}
\multiput(382.00,186.17)(0.500,-1.000){2}{\rule{0.120pt}{0.400pt}}
\put(383,184.67){\rule{0.482pt}{0.400pt}}
\multiput(383.00,185.17)(1.000,-1.000){2}{\rule{0.241pt}{0.400pt}}
\put(385,183.67){\rule{0.482pt}{0.400pt}}
\multiput(385.00,184.17)(1.000,-1.000){2}{\rule{0.241pt}{0.400pt}}
\put(387,182.67){\rule{0.241pt}{0.400pt}}
\multiput(387.00,183.17)(0.500,-1.000){2}{\rule{0.120pt}{0.400pt}}
\put(388,181.67){\rule{0.482pt}{0.400pt}}
\multiput(388.00,182.17)(1.000,-1.000){2}{\rule{0.241pt}{0.400pt}}
\put(390,180.67){\rule{0.241pt}{0.400pt}}
\multiput(390.00,181.17)(0.500,-1.000){2}{\rule{0.120pt}{0.400pt}}
\put(391,179.67){\rule{0.482pt}{0.400pt}}
\multiput(391.00,180.17)(1.000,-1.000){2}{\rule{0.241pt}{0.400pt}}
\put(393,178.67){\rule{0.482pt}{0.400pt}}
\multiput(393.00,179.17)(1.000,-1.000){2}{\rule{0.241pt}{0.400pt}}
\put(395,177.67){\rule{0.241pt}{0.400pt}}
\multiput(395.00,178.17)(0.500,-1.000){2}{\rule{0.120pt}{0.400pt}}
\put(396,176.67){\rule{0.482pt}{0.400pt}}
\multiput(396.00,177.17)(1.000,-1.000){2}{\rule{0.241pt}{0.400pt}}
\put(398,175.67){\rule{0.241pt}{0.400pt}}
\multiput(398.00,176.17)(0.500,-1.000){2}{\rule{0.120pt}{0.400pt}}
\put(399,174.67){\rule{0.482pt}{0.400pt}}
\multiput(399.00,175.17)(1.000,-1.000){2}{\rule{0.241pt}{0.400pt}}
\put(401,173.67){\rule{0.482pt}{0.400pt}}
\multiput(401.00,174.17)(1.000,-1.000){2}{\rule{0.241pt}{0.400pt}}
\put(403,172.67){\rule{0.241pt}{0.400pt}}
\multiput(403.00,173.17)(0.500,-1.000){2}{\rule{0.120pt}{0.400pt}}
\put(406,171.67){\rule{0.241pt}{0.400pt}}
\multiput(406.00,172.17)(0.500,-1.000){2}{\rule{0.120pt}{0.400pt}}
\put(407,170.67){\rule{0.482pt}{0.400pt}}
\multiput(407.00,171.17)(1.000,-1.000){2}{\rule{0.241pt}{0.400pt}}
\put(409,169.67){\rule{0.482pt}{0.400pt}}
\multiput(409.00,170.17)(1.000,-1.000){2}{\rule{0.241pt}{0.400pt}}
\put(411,168.67){\rule{0.241pt}{0.400pt}}
\multiput(411.00,169.17)(0.500,-1.000){2}{\rule{0.120pt}{0.400pt}}
\put(412,167.67){\rule{0.482pt}{0.400pt}}
\multiput(412.00,168.17)(1.000,-1.000){2}{\rule{0.241pt}{0.400pt}}
\put(414,166.67){\rule{0.241pt}{0.400pt}}
\multiput(414.00,167.17)(0.500,-1.000){2}{\rule{0.120pt}{0.400pt}}
\put(415,165.67){\rule{0.482pt}{0.400pt}}
\multiput(415.00,166.17)(1.000,-1.000){2}{\rule{0.241pt}{0.400pt}}
\put(404.0,173.0){\rule[-0.200pt]{0.482pt}{0.400pt}}
\put(419,164.67){\rule{0.241pt}{0.400pt}}
\multiput(419.00,165.17)(0.500,-1.000){2}{\rule{0.120pt}{0.400pt}}
\put(420,163.67){\rule{0.482pt}{0.400pt}}
\multiput(420.00,164.17)(1.000,-1.000){2}{\rule{0.241pt}{0.400pt}}
\put(422,162.67){\rule{0.241pt}{0.400pt}}
\multiput(422.00,163.17)(0.500,-1.000){2}{\rule{0.120pt}{0.400pt}}
\put(423,161.67){\rule{0.482pt}{0.400pt}}
\multiput(423.00,162.17)(1.000,-1.000){2}{\rule{0.241pt}{0.400pt}}
\put(425,160.67){\rule{0.482pt}{0.400pt}}
\multiput(425.00,161.17)(1.000,-1.000){2}{\rule{0.241pt}{0.400pt}}
\put(417.0,166.0){\rule[-0.200pt]{0.482pt}{0.400pt}}
\put(428,159.67){\rule{0.482pt}{0.400pt}}
\multiput(428.00,160.17)(1.000,-1.000){2}{\rule{0.241pt}{0.400pt}}
\put(430,158.67){\rule{0.241pt}{0.400pt}}
\multiput(430.00,159.17)(0.500,-1.000){2}{\rule{0.120pt}{0.400pt}}
\put(431,157.67){\rule{0.482pt}{0.400pt}}
\multiput(431.00,158.17)(1.000,-1.000){2}{\rule{0.241pt}{0.400pt}}
\put(433,156.67){\rule{0.482pt}{0.400pt}}
\multiput(433.00,157.17)(1.000,-1.000){2}{\rule{0.241pt}{0.400pt}}
\put(427.0,161.0){\usebox{\plotpoint}}
\put(436,155.67){\rule{0.482pt}{0.400pt}}
\multiput(436.00,156.17)(1.000,-1.000){2}{\rule{0.241pt}{0.400pt}}
\put(438,154.67){\rule{0.241pt}{0.400pt}}
\multiput(438.00,155.17)(0.500,-1.000){2}{\rule{0.120pt}{0.400pt}}
\put(439,153.67){\rule{0.482pt}{0.400pt}}
\multiput(439.00,154.17)(1.000,-1.000){2}{\rule{0.241pt}{0.400pt}}
\put(435.0,157.0){\usebox{\plotpoint}}
\put(443,152.67){\rule{0.241pt}{0.400pt}}
\multiput(443.00,153.17)(0.500,-1.000){2}{\rule{0.120pt}{0.400pt}}
\put(444,151.67){\rule{0.482pt}{0.400pt}}
\multiput(444.00,152.17)(1.000,-1.000){2}{\rule{0.241pt}{0.400pt}}
\put(446,150.67){\rule{0.241pt}{0.400pt}}
\multiput(446.00,151.17)(0.500,-1.000){2}{\rule{0.120pt}{0.400pt}}
\put(447,149.67){\rule{0.482pt}{0.400pt}}
\multiput(447.00,150.17)(1.000,-1.000){2}{\rule{0.241pt}{0.400pt}}
\put(441.0,154.0){\rule[-0.200pt]{0.482pt}{0.400pt}}
\put(451,148.67){\rule{0.241pt}{0.400pt}}
\multiput(451.00,149.17)(0.500,-1.000){2}{\rule{0.120pt}{0.400pt}}
\put(452,147.67){\rule{0.482pt}{0.400pt}}
\multiput(452.00,148.17)(1.000,-1.000){2}{\rule{0.241pt}{0.400pt}}
\put(449.0,150.0){\rule[-0.200pt]{0.482pt}{0.400pt}}
\put(455,146.67){\rule{0.482pt}{0.400pt}}
\multiput(455.00,147.17)(1.000,-1.000){2}{\rule{0.241pt}{0.400pt}}
\put(457,145.67){\rule{0.482pt}{0.400pt}}
\multiput(457.00,146.17)(1.000,-1.000){2}{\rule{0.241pt}{0.400pt}}
\put(459,144.67){\rule{0.241pt}{0.400pt}}
\multiput(459.00,145.17)(0.500,-1.000){2}{\rule{0.120pt}{0.400pt}}
\put(454.0,148.0){\usebox{\plotpoint}}
\put(462,143.67){\rule{0.241pt}{0.400pt}}
\multiput(462.00,144.17)(0.500,-1.000){2}{\rule{0.120pt}{0.400pt}}
\put(463,142.67){\rule{0.482pt}{0.400pt}}
\multiput(463.00,143.17)(1.000,-1.000){2}{\rule{0.241pt}{0.400pt}}
\put(460.0,145.0){\rule[-0.200pt]{0.482pt}{0.400pt}}
\put(467,141.67){\rule{0.241pt}{0.400pt}}
\multiput(467.00,142.17)(0.500,-1.000){2}{\rule{0.120pt}{0.400pt}}
\put(468,140.67){\rule{0.482pt}{0.400pt}}
\multiput(468.00,141.17)(1.000,-1.000){2}{\rule{0.241pt}{0.400pt}}
\put(470,139.67){\rule{0.241pt}{0.400pt}}
\multiput(470.00,140.17)(0.500,-1.000){2}{\rule{0.120pt}{0.400pt}}
\put(465.0,143.0){\rule[-0.200pt]{0.482pt}{0.400pt}}
\put(473,138.67){\rule{0.482pt}{0.400pt}}
\multiput(473.00,139.17)(1.000,-1.000){2}{\rule{0.241pt}{0.400pt}}
\put(475,137.67){\rule{0.241pt}{0.400pt}}
\multiput(475.00,138.17)(0.500,-1.000){2}{\rule{0.120pt}{0.400pt}}
\put(471.0,140.0){\rule[-0.200pt]{0.482pt}{0.400pt}}
\put(478,136.67){\rule{0.241pt}{0.400pt}}
\multiput(478.00,137.17)(0.500,-1.000){2}{\rule{0.120pt}{0.400pt}}
\put(479,135.67){\rule{0.482pt}{0.400pt}}
\multiput(479.00,136.17)(1.000,-1.000){2}{\rule{0.241pt}{0.400pt}}
\put(476.0,138.0){\rule[-0.200pt]{0.482pt}{0.400pt}}
\put(483,134.67){\rule{0.241pt}{0.400pt}}
\multiput(483.00,135.17)(0.500,-1.000){2}{\rule{0.120pt}{0.400pt}}
\put(484,133.67){\rule{0.482pt}{0.400pt}}
\multiput(484.00,134.17)(1.000,-1.000){2}{\rule{0.241pt}{0.400pt}}
\put(481.0,136.0){\rule[-0.200pt]{0.482pt}{0.400pt}}
\put(487,132.67){\rule{0.482pt}{0.400pt}}
\multiput(487.00,133.17)(1.000,-1.000){2}{\rule{0.241pt}{0.400pt}}
\put(486.0,134.0){\usebox{\plotpoint}}
\put(494,131.67){\rule{0.241pt}{0.400pt}}
\multiput(494.00,132.17)(0.500,-1.000){2}{\rule{0.120pt}{0.400pt}}
\put(489.0,133.0){\rule[-0.200pt]{1.204pt}{0.400pt}}
\put(1310,132.17){\rule{0.482pt}{0.400pt}}
\multiput(1310.00,131.17)(1.000,2.000){2}{\rule{0.241pt}{0.400pt}}
\put(1311.67,134){\rule{0.400pt}{2.168pt}}
\multiput(1311.17,134.00)(1.000,4.500){2}{\rule{0.400pt}{1.084pt}}
\put(1313.17,143){\rule{0.400pt}{0.900pt}}
\multiput(1312.17,143.00)(2.000,2.132){2}{\rule{0.400pt}{0.450pt}}
\put(495.0,132.0){\rule[-0.200pt]{196.333pt}{0.400pt}}
\put(1316.17,147){\rule{0.400pt}{1.100pt}}
\multiput(1315.17,147.00)(2.000,2.717){2}{\rule{0.400pt}{0.550pt}}
\put(1318,152.17){\rule{0.482pt}{0.400pt}}
\multiput(1318.00,151.17)(1.000,2.000){2}{\rule{0.241pt}{0.400pt}}
\put(1319.67,154){\rule{0.400pt}{0.482pt}}
\multiput(1319.17,154.00)(1.000,1.000){2}{\rule{0.400pt}{0.241pt}}
\put(1321,154.17){\rule{0.482pt}{0.400pt}}
\multiput(1321.00,155.17)(1.000,-2.000){2}{\rule{0.241pt}{0.400pt}}
\put(1322.67,154){\rule{0.400pt}{0.964pt}}
\multiput(1322.17,154.00)(1.000,2.000){2}{\rule{0.400pt}{0.482pt}}
\put(1324,157.67){\rule{0.482pt}{0.400pt}}
\multiput(1324.00,157.17)(1.000,1.000){2}{\rule{0.241pt}{0.400pt}}
\put(1326,158.67){\rule{0.482pt}{0.400pt}}
\multiput(1326.00,158.17)(1.000,1.000){2}{\rule{0.241pt}{0.400pt}}
\put(1328,158.67){\rule{0.241pt}{0.400pt}}
\multiput(1328.00,159.17)(0.500,-1.000){2}{\rule{0.120pt}{0.400pt}}
\put(1329.17,133){\rule{0.400pt}{5.300pt}}
\multiput(1328.17,148.00)(2.000,-15.000){2}{\rule{0.400pt}{2.650pt}}
\put(1330.67,133){\rule{0.400pt}{4.818pt}}
\multiput(1330.17,133.00)(1.000,10.000){2}{\rule{0.400pt}{2.409pt}}
\put(1332.17,153){\rule{0.400pt}{2.100pt}}
\multiput(1331.17,153.00)(2.000,5.641){2}{\rule{0.400pt}{1.050pt}}
\put(1334.17,163){\rule{0.400pt}{0.700pt}}
\multiput(1333.17,163.00)(2.000,1.547){2}{\rule{0.400pt}{0.350pt}}
\put(1336,165.67){\rule{0.241pt}{0.400pt}}
\multiput(1336.00,165.17)(0.500,1.000){2}{\rule{0.120pt}{0.400pt}}
\put(1315.0,147.0){\usebox{\plotpoint}}
\put(1338.67,167){\rule{0.400pt}{0.723pt}}
\multiput(1338.17,167.00)(1.000,1.500){2}{\rule{0.400pt}{0.361pt}}
\put(1340,168.67){\rule{0.482pt}{0.400pt}}
\multiput(1340.00,169.17)(1.000,-1.000){2}{\rule{0.241pt}{0.400pt}}
\put(1342.17,169){\rule{0.400pt}{0.700pt}}
\multiput(1341.17,169.00)(2.000,1.547){2}{\rule{0.400pt}{0.350pt}}
\put(1343.67,172){\rule{0.400pt}{0.482pt}}
\multiput(1343.17,172.00)(1.000,1.000){2}{\rule{0.400pt}{0.241pt}}
\put(1345.17,174){\rule{0.400pt}{0.700pt}}
\multiput(1344.17,174.00)(2.000,1.547){2}{\rule{0.400pt}{0.350pt}}
\put(1347,176.67){\rule{0.241pt}{0.400pt}}
\multiput(1347.00,176.17)(0.500,1.000){2}{\rule{0.120pt}{0.400pt}}
\put(1337.0,167.0){\rule[-0.200pt]{0.482pt}{0.400pt}}
\put(1352,177.67){\rule{0.241pt}{0.400pt}}
\multiput(1352.00,177.17)(0.500,1.000){2}{\rule{0.120pt}{0.400pt}}
\put(1353,177.17){\rule{0.482pt}{0.400pt}}
\multiput(1353.00,178.17)(1.000,-2.000){2}{\rule{0.241pt}{0.400pt}}
\put(1354.67,177){\rule{0.400pt}{0.482pt}}
\multiput(1354.17,177.00)(1.000,1.000){2}{\rule{0.400pt}{0.241pt}}
\put(1348.0,178.0){\rule[-0.200pt]{0.964pt}{0.400pt}}
\put(1359.67,179){\rule{0.400pt}{0.723pt}}
\multiput(1359.17,179.00)(1.000,1.500){2}{\rule{0.400pt}{0.361pt}}
\put(1361,180.17){\rule{0.482pt}{0.400pt}}
\multiput(1361.00,181.17)(1.000,-2.000){2}{\rule{0.241pt}{0.400pt}}
\put(1362.67,180){\rule{0.400pt}{0.964pt}}
\multiput(1362.17,180.00)(1.000,2.000){2}{\rule{0.400pt}{0.482pt}}
\put(1364.17,136){\rule{0.400pt}{9.700pt}}
\multiput(1363.17,163.87)(2.000,-27.867){2}{\rule{0.400pt}{4.850pt}}
\put(1366.17,136){\rule{0.400pt}{6.100pt}}
\multiput(1365.17,136.00)(2.000,17.339){2}{\rule{0.400pt}{3.050pt}}
\put(1367.67,166){\rule{0.400pt}{3.373pt}}
\multiput(1367.17,166.00)(1.000,7.000){2}{\rule{0.400pt}{1.686pt}}
\put(1369,180.17){\rule{0.482pt}{0.400pt}}
\multiput(1369.00,179.17)(1.000,2.000){2}{\rule{0.241pt}{0.400pt}}
\put(1371,181.67){\rule{0.241pt}{0.400pt}}
\multiput(1371.00,181.17)(0.500,1.000){2}{\rule{0.120pt}{0.400pt}}
\put(1372,183.17){\rule{0.482pt}{0.400pt}}
\multiput(1372.00,182.17)(1.000,2.000){2}{\rule{0.241pt}{0.400pt}}
\put(1374,184.67){\rule{0.482pt}{0.400pt}}
\multiput(1374.00,184.17)(1.000,1.000){2}{\rule{0.241pt}{0.400pt}}
\put(1375.67,177){\rule{0.400pt}{2.168pt}}
\multiput(1375.17,181.50)(1.000,-4.500){2}{\rule{0.400pt}{1.084pt}}
\put(1377.17,177){\rule{0.400pt}{2.100pt}}
\multiput(1376.17,177.00)(2.000,5.641){2}{\rule{0.400pt}{1.050pt}}
\put(1379,185.67){\rule{0.241pt}{0.400pt}}
\multiput(1379.00,186.17)(0.500,-1.000){2}{\rule{0.120pt}{0.400pt}}
\put(1380,185.67){\rule{0.482pt}{0.400pt}}
\multiput(1380.00,185.17)(1.000,1.000){2}{\rule{0.241pt}{0.400pt}}
\put(1382,186.67){\rule{0.482pt}{0.400pt}}
\multiput(1382.00,186.17)(1.000,1.000){2}{\rule{0.241pt}{0.400pt}}
\put(1383.67,188){\rule{0.400pt}{0.482pt}}
\multiput(1383.17,188.00)(1.000,1.000){2}{\rule{0.400pt}{0.241pt}}
\put(1385,188.67){\rule{0.482pt}{0.400pt}}
\multiput(1385.00,189.17)(1.000,-1.000){2}{\rule{0.241pt}{0.400pt}}
\put(1387,188.67){\rule{0.241pt}{0.400pt}}
\multiput(1387.00,188.17)(0.500,1.000){2}{\rule{0.120pt}{0.400pt}}
\put(1388,188.67){\rule{0.482pt}{0.400pt}}
\multiput(1388.00,189.17)(1.000,-1.000){2}{\rule{0.241pt}{0.400pt}}
\put(1356.0,179.0){\rule[-0.200pt]{0.964pt}{0.400pt}}
\put(1391,187.67){\rule{0.482pt}{0.400pt}}
\multiput(1391.00,188.17)(1.000,-1.000){2}{\rule{0.241pt}{0.400pt}}
\put(1393.17,133){\rule{0.400pt}{11.100pt}}
\multiput(1392.17,164.96)(2.000,-31.961){2}{\rule{0.400pt}{5.550pt}}
\put(1390.0,189.0){\usebox{\plotpoint}}
\put(1398,132.67){\rule{0.241pt}{0.400pt}}
\multiput(1398.00,132.17)(0.500,1.000){2}{\rule{0.120pt}{0.400pt}}
\put(1395.0,133.0){\rule[-0.200pt]{0.723pt}{0.400pt}}
\put(1401.17,134){\rule{0.400pt}{0.700pt}}
\multiput(1400.17,134.00)(2.000,1.547){2}{\rule{0.400pt}{0.350pt}}
\put(1402.67,137){\rule{0.400pt}{10.600pt}}
\multiput(1402.17,137.00)(1.000,22.000){2}{\rule{0.400pt}{5.300pt}}
\put(1404.17,181){\rule{0.400pt}{1.300pt}}
\multiput(1403.17,181.00)(2.000,3.302){2}{\rule{0.400pt}{0.650pt}}
\put(1405.67,187){\rule{0.400pt}{0.723pt}}
\multiput(1405.17,187.00)(1.000,1.500){2}{\rule{0.400pt}{0.361pt}}
\put(1407,190.17){\rule{0.482pt}{0.400pt}}
\multiput(1407.00,189.17)(1.000,2.000){2}{\rule{0.241pt}{0.400pt}}
\put(1399.0,134.0){\rule[-0.200pt]{0.482pt}{0.400pt}}
\put(1410.67,192){\rule{0.400pt}{0.482pt}}
\multiput(1410.17,192.00)(1.000,1.000){2}{\rule{0.400pt}{0.241pt}}
\put(1412.17,194){\rule{0.400pt}{0.900pt}}
\multiput(1411.17,194.00)(2.000,2.132){2}{\rule{0.400pt}{0.450pt}}
\put(1409.0,192.0){\rule[-0.200pt]{0.482pt}{0.400pt}}
\put(1418.67,196){\rule{0.400pt}{0.482pt}}
\multiput(1418.17,197.00)(1.000,-1.000){2}{\rule{0.400pt}{0.241pt}}
\put(1420.17,196){\rule{0.400pt}{0.700pt}}
\multiput(1419.17,196.00)(2.000,1.547){2}{\rule{0.400pt}{0.350pt}}
\put(1421.67,199){\rule{0.400pt}{0.964pt}}
\multiput(1421.17,199.00)(1.000,2.000){2}{\rule{0.400pt}{0.482pt}}
\put(1423,201.17){\rule{0.482pt}{0.400pt}}
\multiput(1423.00,202.17)(1.000,-2.000){2}{\rule{0.241pt}{0.400pt}}
\put(1425.17,201){\rule{0.400pt}{0.900pt}}
\multiput(1424.17,201.00)(2.000,2.132){2}{\rule{0.400pt}{0.450pt}}
\put(1414.0,198.0){\rule[-0.200pt]{1.204pt}{0.400pt}}
\put(1428,204.67){\rule{0.482pt}{0.400pt}}
\multiput(1428.00,204.17)(1.000,1.000){2}{\rule{0.241pt}{0.400pt}}
\put(1430,204.67){\rule{0.241pt}{0.400pt}}
\multiput(1430.00,205.17)(0.500,-1.000){2}{\rule{0.120pt}{0.400pt}}
\put(1431.17,205){\rule{0.400pt}{0.900pt}}
\multiput(1430.17,205.00)(2.000,2.132){2}{\rule{0.400pt}{0.450pt}}
\put(1433,207.17){\rule{0.482pt}{0.400pt}}
\multiput(1433.00,208.17)(1.000,-2.000){2}{\rule{0.241pt}{0.400pt}}
\put(1435,205.67){\rule{0.241pt}{0.400pt}}
\multiput(1435.00,206.17)(0.500,-1.000){2}{\rule{0.120pt}{0.400pt}}
\put(1427.0,205.0){\usebox{\plotpoint}}
\put(1437.67,206){\rule{0.400pt}{0.964pt}}
\multiput(1437.17,206.00)(1.000,2.000){2}{\rule{0.400pt}{0.482pt}}
\put(1439,209.67){\rule{0.482pt}{0.400pt}}
\multiput(1439.00,209.17)(1.000,1.000){2}{\rule{0.241pt}{0.400pt}}
\put(1441,210.67){\rule{0.482pt}{0.400pt}}
\multiput(1441.00,210.17)(1.000,1.000){2}{\rule{0.241pt}{0.400pt}}
\put(1442.67,212){\rule{0.400pt}{0.482pt}}
\multiput(1442.17,212.00)(1.000,1.000){2}{\rule{0.400pt}{0.241pt}}
\put(1444.17,214){\rule{0.400pt}{0.700pt}}
\multiput(1443.17,214.00)(2.000,1.547){2}{\rule{0.400pt}{0.350pt}}
\put(1446,216.67){\rule{0.241pt}{0.400pt}}
\multiput(1446.00,216.17)(0.500,1.000){2}{\rule{0.120pt}{0.400pt}}
\put(1447.17,218){\rule{0.400pt}{1.100pt}}
\multiput(1446.17,218.00)(2.000,2.717){2}{\rule{0.400pt}{0.550pt}}
\put(1436.0,206.0){\rule[-0.200pt]{0.482pt}{0.400pt}}
\put(173,827){\makebox(0,0){$+$}}
\put(199,452){\makebox(0,0){$+$}}
\put(225,365){\makebox(0,0){$+$}}
\put(251,313){\makebox(0,0){$+$}}
\put(277,276){\makebox(0,0){$+$}}
\put(303,247){\makebox(0,0){$+$}}
\put(329,224){\makebox(0,0){$+$}}
\put(355,203){\makebox(0,0){$+$}}
\put(381,186){\makebox(0,0){$+$}}
\put(407,171){\makebox(0,0){$+$}}
\put(433,157){\makebox(0,0){$+$}}
\put(459,145){\makebox(0,0){$+$}}
\put(485,133){\makebox(0,0){$+$}}
\put(171.0,131.0){\rule[-0.200pt]{0.400pt}{175.375pt}}
\put(171.0,131.0){\rule[-0.200pt]{307.870pt}{0.400pt}}
\put(1449.0,131.0){\rule[-0.200pt]{0.400pt}{175.375pt}}
\put(171.0,859.0){\rule[-0.200pt]{307.870pt}{0.400pt}}
\end{picture}

%% file: figure4.tex
\setlength{\unitlength}{0.240900pt}
\ifx\plotpoint\undefined\newsavebox{\plotpoint}\fi
\begin{picture}(1500,900)(0,0)
\sbox{\plotpoint}{\rule[-0.200pt]{0.400pt}{0.400pt}}%
\put(171.0,131.0){\rule[-0.200pt]{4.818pt}{0.400pt}}
\put(151,131){\makebox(0,0)[r]{$0.0$}}
\put(1429.0,131.0){\rule[-0.200pt]{4.818pt}{0.400pt}}
\put(171.0,277.0){\rule[-0.200pt]{4.818pt}{0.400pt}}
\put(151,277){\makebox(0,0)[r]{$0.2$}}
\put(1429.0,277.0){\rule[-0.200pt]{4.818pt}{0.400pt}}
\put(171.0,422.0){\rule[-0.200pt]{4.818pt}{0.400pt}}
\put(151,422){\makebox(0,0)[r]{$0.4$}}
\put(1429.0,422.0){\rule[-0.200pt]{4.818pt}{0.400pt}}
\put(171.0,568.0){\rule[-0.200pt]{4.818pt}{0.400pt}}
\put(151,568){\makebox(0,0)[r]{$0.6$}}
\put(1429.0,568.0){\rule[-0.200pt]{4.818pt}{0.400pt}}
\put(171.0,713.0){\rule[-0.200pt]{4.818pt}{0.400pt}}
\put(151,713){\makebox(0,0)[r]{$0.8$}}
\put(1429.0,713.0){\rule[-0.200pt]{4.818pt}{0.400pt}}
\put(171.0,859.0){\rule[-0.200pt]{4.818pt}{0.400pt}}
\put(151,859){\makebox(0,0)[r]{$1.0$}}
\put(1429.0,859.0){\rule[-0.200pt]{4.818pt}{0.400pt}}
\put(171.0,131.0){\rule[-0.200pt]{0.400pt}{4.818pt}}
\put(171,90){\makebox(0,0){$0.0$}}
\put(171.0,839.0){\rule[-0.200pt]{0.400pt}{4.818pt}}
\put(331.0,131.0){\rule[-0.200pt]{0.400pt}{4.818pt}}
\put(331,90){\makebox(0,0){$0.5$}}
\put(331.0,839.0){\rule[-0.200pt]{0.400pt}{4.818pt}}
\put(491.0,131.0){\rule[-0.200pt]{0.400pt}{4.818pt}}
\put(491,90){\makebox(0,0){$1.0$}}
\put(491.0,839.0){\rule[-0.200pt]{0.400pt}{4.818pt}}
\put(650.0,131.0){\rule[-0.200pt]{0.400pt}{4.818pt}}
\put(650,90){\makebox(0,0){$1.5$}}
\put(650.0,839.0){\rule[-0.200pt]{0.400pt}{4.818pt}}
\put(810.0,131.0){\rule[-0.200pt]{0.400pt}{4.818pt}}
\put(810,90){\makebox(0,0){$2.0$}}
\put(810.0,839.0){\rule[-0.200pt]{0.400pt}{4.818pt}}
\put(970.0,131.0){\rule[-0.200pt]{0.400pt}{4.818pt}}
\put(970,90){\makebox(0,0){$2.5$}}
\put(970.0,839.0){\rule[-0.200pt]{0.400pt}{4.818pt}}
\put(1130.0,131.0){\rule[-0.200pt]{0.400pt}{4.818pt}}
\put(1130,90){\makebox(0,0){$3.0$}}
\put(1130.0,839.0){\rule[-0.200pt]{0.400pt}{4.818pt}}
\put(1289.0,131.0){\rule[-0.200pt]{0.400pt}{4.818pt}}
\put(1289,90){\makebox(0,0){$3.5$}}
\put(1289.0,839.0){\rule[-0.200pt]{0.400pt}{4.818pt}}
\put(1449.0,131.0){\rule[-0.200pt]{0.400pt}{4.818pt}}
\put(1449,90){\makebox(0,0){$4.0$}}
\put(1449.0,839.0){\rule[-0.200pt]{0.400pt}{4.818pt}}
\put(171.0,131.0){\rule[-0.200pt]{0.400pt}{175.375pt}}
\put(171.0,131.0){\rule[-0.200pt]{307.870pt}{0.400pt}}
\put(1449.0,131.0){\rule[-0.200pt]{0.400pt}{175.375pt}}
\put(171.0,859.0){\rule[-0.200pt]{307.870pt}{0.400pt}}
\put(50,495){\makebox(0,0){\rotatebox{90}{$\sigma_{est}$}}}
\put(810,29){\makebox(0,0){$\mu$}}
\put(173,143){\usebox{\plotpoint}}
\put(491,141.67){\rule{0.241pt}{0.400pt}}
\multiput(491.00,142.17)(0.500,-1.000){2}{\rule{0.120pt}{0.400pt}}
\put(173.0,143.0){\rule[-0.200pt]{76.606pt}{0.400pt}}
\put(498,140.67){\rule{0.482pt}{0.400pt}}
\multiput(498.00,141.17)(1.000,-1.000){2}{\rule{0.241pt}{0.400pt}}
\put(492.0,142.0){\rule[-0.200pt]{1.445pt}{0.400pt}}
\put(522,139.67){\rule{0.482pt}{0.400pt}}
\multiput(522.00,140.17)(1.000,-1.000){2}{\rule{0.241pt}{0.400pt}}
\put(500.0,141.0){\rule[-0.200pt]{5.300pt}{0.400pt}}
\put(768,138.67){\rule{0.482pt}{0.400pt}}
\multiput(768.00,139.17)(1.000,-1.000){2}{\rule{0.241pt}{0.400pt}}
\put(770,138.67){\rule{0.482pt}{0.400pt}}
\multiput(770.00,138.17)(1.000,1.000){2}{\rule{0.241pt}{0.400pt}}
\put(524.0,140.0){\rule[-0.200pt]{58.780pt}{0.400pt}}
\put(808,138.67){\rule{0.482pt}{0.400pt}}
\multiput(808.00,139.17)(1.000,-1.000){2}{\rule{0.241pt}{0.400pt}}
\put(810,138.67){\rule{0.482pt}{0.400pt}}
\multiput(810.00,138.17)(1.000,1.000){2}{\rule{0.241pt}{0.400pt}}
\put(772.0,140.0){\rule[-0.200pt]{8.672pt}{0.400pt}}
\put(848,138.67){\rule{0.482pt}{0.400pt}}
\multiput(848.00,139.17)(1.000,-1.000){2}{\rule{0.241pt}{0.400pt}}
\put(850,138.67){\rule{0.482pt}{0.400pt}}
\multiput(850.00,138.17)(1.000,1.000){2}{\rule{0.241pt}{0.400pt}}
\put(812.0,140.0){\rule[-0.200pt]{8.672pt}{0.400pt}}
\put(1085,139.67){\rule{0.241pt}{0.400pt}}
\multiput(1085.00,139.17)(0.500,1.000){2}{\rule{0.120pt}{0.400pt}}
\put(852.0,140.0){\rule[-0.200pt]{56.130pt}{0.400pt}}
\put(1117,140.67){\rule{0.241pt}{0.400pt}}
\multiput(1117.00,140.17)(0.500,1.000){2}{\rule{0.120pt}{0.400pt}}
\put(1086.0,141.0){\rule[-0.200pt]{7.468pt}{0.400pt}}
\put(1128,141.67){\rule{0.482pt}{0.400pt}}
\multiput(1128.00,141.17)(1.000,1.000){2}{\rule{0.241pt}{0.400pt}}
\put(1130,141.67){\rule{0.241pt}{0.400pt}}
\multiput(1130.00,142.17)(0.500,-1.000){2}{\rule{0.120pt}{0.400pt}}
\put(1118.0,142.0){\rule[-0.200pt]{2.409pt}{0.400pt}}
\put(1136,140.67){\rule{0.241pt}{0.400pt}}
\multiput(1136.00,141.17)(0.500,-1.000){2}{\rule{0.120pt}{0.400pt}}
\put(1131.0,142.0){\rule[-0.200pt]{1.204pt}{0.400pt}}
\put(1161,139.67){\rule{0.482pt}{0.400pt}}
\multiput(1161.00,140.17)(1.000,-1.000){2}{\rule{0.241pt}{0.400pt}}
\put(1137.0,141.0){\rule[-0.200pt]{5.782pt}{0.400pt}}
\put(1235,139.67){\rule{0.482pt}{0.400pt}}
\multiput(1235.00,139.17)(1.000,1.000){2}{\rule{0.241pt}{0.400pt}}
\put(1163.0,140.0){\rule[-0.200pt]{17.345pt}{0.400pt}}
\put(1262,140.67){\rule{0.482pt}{0.400pt}}
\multiput(1262.00,140.17)(1.000,1.000){2}{\rule{0.241pt}{0.400pt}}
\put(1237.0,141.0){\rule[-0.200pt]{6.022pt}{0.400pt}}
\put(1272,141.67){\rule{0.241pt}{0.400pt}}
\multiput(1272.00,141.17)(0.500,1.000){2}{\rule{0.120pt}{0.400pt}}
\put(1273,141.67){\rule{0.482pt}{0.400pt}}
\multiput(1273.00,142.17)(1.000,-1.000){2}{\rule{0.241pt}{0.400pt}}
\put(1264.0,142.0){\rule[-0.200pt]{1.927pt}{0.400pt}}
\put(1278,140.67){\rule{0.482pt}{0.400pt}}
\multiput(1278.00,141.17)(1.000,-1.000){2}{\rule{0.241pt}{0.400pt}}
\put(1275.0,142.0){\rule[-0.200pt]{0.723pt}{0.400pt}}
\put(1296,140.67){\rule{0.241pt}{0.400pt}}
\multiput(1296.00,140.17)(0.500,1.000){2}{\rule{0.120pt}{0.400pt}}
\put(1280.0,141.0){\rule[-0.200pt]{3.854pt}{0.400pt}}
\put(1302,141.67){\rule{0.482pt}{0.400pt}}
\multiput(1302.00,141.17)(1.000,1.000){2}{\rule{0.241pt}{0.400pt}}
\put(1304,141.67){\rule{0.241pt}{0.400pt}}
\multiput(1304.00,142.17)(0.500,-1.000){2}{\rule{0.120pt}{0.400pt}}
\put(1297.0,142.0){\rule[-0.200pt]{1.204pt}{0.400pt}}
\put(1308,141.67){\rule{0.482pt}{0.400pt}}
\multiput(1308.00,141.17)(1.000,1.000){2}{\rule{0.241pt}{0.400pt}}
\put(1310.17,143){\rule{0.400pt}{2.700pt}}
\multiput(1309.17,143.00)(2.000,7.396){2}{\rule{0.400pt}{1.350pt}}
\put(1311.67,156){\rule{0.400pt}{17.104pt}}
\multiput(1311.17,156.00)(1.000,35.500){2}{\rule{0.400pt}{8.552pt}}
\put(1313.17,227){\rule{0.400pt}{6.100pt}}
\multiput(1312.17,227.00)(2.000,17.339){2}{\rule{0.400pt}{3.050pt}}
\put(1314.67,257){\rule{0.400pt}{0.964pt}}
\multiput(1314.17,257.00)(1.000,2.000){2}{\rule{0.400pt}{0.482pt}}
\put(1316.17,261){\rule{0.400pt}{7.300pt}}
\multiput(1315.17,261.00)(2.000,20.848){2}{\rule{0.400pt}{3.650pt}}
\put(1318.17,297){\rule{0.400pt}{3.900pt}}
\multiput(1317.17,297.00)(2.000,10.905){2}{\rule{0.400pt}{1.950pt}}
\put(1319.67,316){\rule{0.400pt}{3.854pt}}
\multiput(1319.17,316.00)(1.000,8.000){2}{\rule{0.400pt}{1.927pt}}
\put(1321.17,313){\rule{0.400pt}{3.900pt}}
\multiput(1320.17,323.91)(2.000,-10.905){2}{\rule{0.400pt}{1.950pt}}
\put(1322.67,313){\rule{0.400pt}{8.672pt}}
\multiput(1322.17,313.00)(1.000,18.000){2}{\rule{0.400pt}{4.336pt}}
\put(1324.17,349){\rule{0.400pt}{0.900pt}}
\multiput(1323.17,349.00)(2.000,2.132){2}{\rule{0.400pt}{0.450pt}}
\put(1326.17,353){\rule{0.400pt}{2.700pt}}
\multiput(1325.17,353.00)(2.000,7.396){2}{\rule{0.400pt}{1.350pt}}
\put(1327.67,357){\rule{0.400pt}{2.168pt}}
\multiput(1327.17,361.50)(1.000,-4.500){2}{\rule{0.400pt}{1.084pt}}
\put(1329.17,147){\rule{0.400pt}{42.100pt}}
\multiput(1328.17,269.62)(2.000,-122.619){2}{\rule{0.400pt}{21.050pt}}
\put(1330.67,147){\rule{0.400pt}{38.785pt}}
\multiput(1330.17,147.00)(1.000,80.500){2}{\rule{0.400pt}{19.392pt}}
\put(1332.17,308){\rule{0.400pt}{16.100pt}}
\multiput(1331.17,308.00)(2.000,46.584){2}{\rule{0.400pt}{8.050pt}}
\put(1334.17,388){\rule{0.400pt}{4.900pt}}
\multiput(1333.17,388.00)(2.000,13.830){2}{\rule{0.400pt}{2.450pt}}
\put(1335.67,412){\rule{0.400pt}{1.445pt}}
\multiput(1335.17,412.00)(1.000,3.000){2}{\rule{0.400pt}{0.723pt}}
\put(1337,417.67){\rule{0.482pt}{0.400pt}}
\multiput(1337.00,417.17)(1.000,1.000){2}{\rule{0.241pt}{0.400pt}}
\put(1338.67,419){\rule{0.400pt}{6.263pt}}
\multiput(1338.17,419.00)(1.000,13.000){2}{\rule{0.400pt}{3.132pt}}
\put(1340.17,438){\rule{0.400pt}{1.500pt}}
\multiput(1339.17,441.89)(2.000,-3.887){2}{\rule{0.400pt}{0.750pt}}
\put(1342.17,438){\rule{0.400pt}{4.500pt}}
\multiput(1341.17,438.00)(2.000,12.660){2}{\rule{0.400pt}{2.250pt}}
\put(1343.67,460){\rule{0.400pt}{3.132pt}}
\multiput(1343.17,460.00)(1.000,6.500){2}{\rule{0.400pt}{1.566pt}}
\put(1345.17,473){\rule{0.400pt}{5.900pt}}
\multiput(1344.17,473.00)(2.000,16.754){2}{\rule{0.400pt}{2.950pt}}
\put(1346.67,502){\rule{0.400pt}{1.204pt}}
\multiput(1346.17,502.00)(1.000,2.500){2}{\rule{0.400pt}{0.602pt}}
\put(1348,507.17){\rule{0.482pt}{0.400pt}}
\multiput(1348.00,506.17)(1.000,2.000){2}{\rule{0.241pt}{0.400pt}}
\put(1350,507.17){\rule{0.482pt}{0.400pt}}
\multiput(1350.00,508.17)(1.000,-2.000){2}{\rule{0.241pt}{0.400pt}}
\put(1351.67,507){\rule{0.400pt}{1.927pt}}
\multiput(1351.17,507.00)(1.000,4.000){2}{\rule{0.400pt}{0.964pt}}
\put(1353.17,496){\rule{0.400pt}{3.900pt}}
\multiput(1352.17,506.91)(2.000,-10.905){2}{\rule{0.400pt}{1.950pt}}
\put(1354.67,496){\rule{0.400pt}{4.818pt}}
\multiput(1354.17,496.00)(1.000,10.000){2}{\rule{0.400pt}{2.409pt}}
\put(1356,516.17){\rule{0.482pt}{0.400pt}}
\multiput(1356.00,515.17)(1.000,2.000){2}{\rule{0.241pt}{0.400pt}}
\put(1358.17,515){\rule{0.400pt}{0.700pt}}
\multiput(1357.17,516.55)(2.000,-1.547){2}{\rule{0.400pt}{0.350pt}}
\put(1359.67,515){\rule{0.400pt}{5.059pt}}
\multiput(1359.17,515.00)(1.000,10.500){2}{\rule{0.400pt}{2.529pt}}
\put(1361.17,524){\rule{0.400pt}{2.500pt}}
\multiput(1360.17,530.81)(2.000,-6.811){2}{\rule{0.400pt}{1.250pt}}
\put(1362.67,524){\rule{0.400pt}{8.191pt}}
\multiput(1362.17,524.00)(1.000,17.000){2}{\rule{0.400pt}{4.095pt}}
\put(1364.17,168){\rule{0.400pt}{78.100pt}}
\multiput(1363.17,395.90)(2.000,-227.900){2}{\rule{0.400pt}{39.050pt}}
\put(1366.17,168){\rule{0.400pt}{48.700pt}}
\multiput(1365.17,168.00)(2.000,141.921){2}{\rule{0.400pt}{24.350pt}}
\put(1367.67,411){\rule{0.400pt}{26.981pt}}
\multiput(1367.17,411.00)(1.000,56.000){2}{\rule{0.400pt}{13.490pt}}
\put(1369.17,523){\rule{0.400pt}{2.700pt}}
\multiput(1368.17,523.00)(2.000,7.396){2}{\rule{0.400pt}{1.350pt}}
\put(1370.67,536){\rule{0.400pt}{3.614pt}}
\multiput(1370.17,536.00)(1.000,7.500){2}{\rule{0.400pt}{1.807pt}}
\put(1372.17,551){\rule{0.400pt}{3.100pt}}
\multiput(1371.17,551.00)(2.000,8.566){2}{\rule{0.400pt}{1.550pt}}
\put(1374.17,566){\rule{0.400pt}{1.100pt}}
\multiput(1373.17,566.00)(2.000,2.717){2}{\rule{0.400pt}{0.550pt}}
\put(1375.67,497){\rule{0.400pt}{17.827pt}}
\multiput(1375.17,534.00)(1.000,-37.000){2}{\rule{0.400pt}{8.913pt}}
\put(1377.17,497){\rule{0.400pt}{16.500pt}}
\multiput(1376.17,497.00)(2.000,47.753){2}{\rule{0.400pt}{8.250pt}}
\put(1378.67,574){\rule{0.400pt}{1.204pt}}
\multiput(1378.17,576.50)(1.000,-2.500){2}{\rule{0.400pt}{0.602pt}}
\put(1380.17,574){\rule{0.400pt}{0.700pt}}
\multiput(1379.17,574.00)(2.000,1.547){2}{\rule{0.400pt}{0.350pt}}
\put(1382.17,577){\rule{0.400pt}{2.700pt}}
\multiput(1381.17,577.00)(2.000,7.396){2}{\rule{0.400pt}{1.350pt}}
\put(1383.67,590){\rule{0.400pt}{2.650pt}}
\multiput(1383.17,590.00)(1.000,5.500){2}{\rule{0.400pt}{1.325pt}}
\put(1385.17,594){\rule{0.400pt}{1.500pt}}
\multiput(1384.17,597.89)(2.000,-3.887){2}{\rule{0.400pt}{0.750pt}}
\put(1386.67,594){\rule{0.400pt}{1.927pt}}
\multiput(1386.17,594.00)(1.000,4.000){2}{\rule{0.400pt}{0.964pt}}
\put(1388.17,595){\rule{0.400pt}{1.500pt}}
\multiput(1387.17,598.89)(2.000,-3.887){2}{\rule{0.400pt}{0.750pt}}
\put(1390,593.67){\rule{0.241pt}{0.400pt}}
\multiput(1390.00,594.17)(0.500,-1.000){2}{\rule{0.120pt}{0.400pt}}
\put(1391.17,589){\rule{0.400pt}{1.100pt}}
\multiput(1390.17,591.72)(2.000,-2.717){2}{\rule{0.400pt}{0.550pt}}
\put(1393.17,146){\rule{0.400pt}{88.700pt}}
\multiput(1392.17,404.90)(2.000,-258.899){2}{\rule{0.400pt}{44.350pt}}
\put(1305.0,142.0){\rule[-0.200pt]{0.723pt}{0.400pt}}
\put(1396,144.67){\rule{0.482pt}{0.400pt}}
\multiput(1396.00,145.17)(1.000,-1.000){2}{\rule{0.241pt}{0.400pt}}
\put(1397.67,145){\rule{0.400pt}{1.927pt}}
\multiput(1397.17,145.00)(1.000,4.000){2}{\rule{0.400pt}{0.964pt}}
\put(1399.17,153){\rule{0.400pt}{0.700pt}}
\multiput(1398.17,153.00)(2.000,1.547){2}{\rule{0.400pt}{0.350pt}}
\put(1401.17,156){\rule{0.400pt}{5.500pt}}
\multiput(1400.17,156.00)(2.000,15.584){2}{\rule{0.400pt}{2.750pt}}
\put(1402.67,183){\rule{0.400pt}{84.556pt}}
\multiput(1402.17,183.00)(1.000,175.500){2}{\rule{0.400pt}{42.278pt}}
\put(1404.17,534){\rule{0.400pt}{9.900pt}}
\multiput(1403.17,534.00)(2.000,28.452){2}{\rule{0.400pt}{4.950pt}}
\put(1405.67,583){\rule{0.400pt}{4.818pt}}
\multiput(1405.17,583.00)(1.000,10.000){2}{\rule{0.400pt}{2.409pt}}
\put(1407.17,603){\rule{0.400pt}{3.300pt}}
\multiput(1406.17,603.00)(2.000,9.151){2}{\rule{0.400pt}{1.650pt}}
\put(1409,617.67){\rule{0.482pt}{0.400pt}}
\multiput(1409.00,618.17)(1.000,-1.000){2}{\rule{0.241pt}{0.400pt}}
\put(1410.67,618){\rule{0.400pt}{4.336pt}}
\multiput(1410.17,618.00)(1.000,9.000){2}{\rule{0.400pt}{2.168pt}}
\put(1412.17,636){\rule{0.400pt}{7.300pt}}
\multiput(1411.17,636.00)(2.000,20.848){2}{\rule{0.400pt}{3.650pt}}
\put(1413.67,665){\rule{0.400pt}{1.686pt}}
\multiput(1413.17,668.50)(1.000,-3.500){2}{\rule{0.400pt}{0.843pt}}
\put(1415,663.67){\rule{0.482pt}{0.400pt}}
\multiput(1415.00,664.17)(1.000,-1.000){2}{\rule{0.241pt}{0.400pt}}
\put(1417.17,664){\rule{0.400pt}{0.700pt}}
\multiput(1416.17,664.00)(2.000,1.547){2}{\rule{0.400pt}{0.350pt}}
\put(1418.67,649){\rule{0.400pt}{4.336pt}}
\multiput(1418.17,658.00)(1.000,-9.000){2}{\rule{0.400pt}{2.168pt}}
\put(1420.17,649){\rule{0.400pt}{5.500pt}}
\multiput(1419.17,649.00)(2.000,15.584){2}{\rule{0.400pt}{2.750pt}}
\put(1421.67,676){\rule{0.400pt}{8.191pt}}
\multiput(1421.17,676.00)(1.000,17.000){2}{\rule{0.400pt}{4.095pt}}
\put(1423.17,692){\rule{0.400pt}{3.700pt}}
\multiput(1422.17,702.32)(2.000,-10.320){2}{\rule{0.400pt}{1.850pt}}
\put(1425.17,692){\rule{0.400pt}{7.100pt}}
\multiput(1424.17,692.00)(2.000,20.264){2}{\rule{0.400pt}{3.550pt}}
\put(1427,726.67){\rule{0.241pt}{0.400pt}}
\multiput(1427.00,726.17)(0.500,1.000){2}{\rule{0.120pt}{0.400pt}}
\put(1428.17,728){\rule{0.400pt}{0.900pt}}
\multiput(1427.17,728.00)(2.000,2.132){2}{\rule{0.400pt}{0.450pt}}
\put(1429.67,720){\rule{0.400pt}{2.891pt}}
\multiput(1429.17,726.00)(1.000,-6.000){2}{\rule{0.400pt}{1.445pt}}
\put(1431.17,720){\rule{0.400pt}{7.100pt}}
\multiput(1430.17,720.00)(2.000,20.264){2}{\rule{0.400pt}{3.550pt}}
\put(1433.17,741){\rule{0.400pt}{2.900pt}}
\multiput(1432.17,748.98)(2.000,-7.981){2}{\rule{0.400pt}{1.450pt}}
\put(1434.67,733){\rule{0.400pt}{1.927pt}}
\multiput(1434.17,737.00)(1.000,-4.000){2}{\rule{0.400pt}{0.964pt}}
\put(1436.17,729){\rule{0.400pt}{0.900pt}}
\multiput(1435.17,731.13)(2.000,-2.132){2}{\rule{0.400pt}{0.450pt}}
\put(1437.67,729){\rule{0.400pt}{7.709pt}}
\multiput(1437.17,729.00)(1.000,16.000){2}{\rule{0.400pt}{3.854pt}}
\put(1439.17,761){\rule{0.400pt}{2.300pt}}
\multiput(1438.17,761.00)(2.000,6.226){2}{\rule{0.400pt}{1.150pt}}
\put(1441.17,772){\rule{0.400pt}{2.100pt}}
\multiput(1440.17,772.00)(2.000,5.641){2}{\rule{0.400pt}{1.050pt}}
\put(1442.67,782){\rule{0.400pt}{2.409pt}}
\multiput(1442.17,782.00)(1.000,5.000){2}{\rule{0.400pt}{1.204pt}}
\put(1444.17,792){\rule{0.400pt}{5.500pt}}
\multiput(1443.17,792.00)(2.000,15.584){2}{\rule{0.400pt}{2.750pt}}
\put(1445.67,819){\rule{0.400pt}{1.927pt}}
\multiput(1445.17,819.00)(1.000,4.000){2}{\rule{0.400pt}{0.964pt}}
\put(1447.17,827){\rule{0.400pt}{6.500pt}}
\multiput(1446.17,827.00)(2.000,18.509){2}{\rule{0.400pt}{3.250pt}}
\put(1395.0,146.0){\usebox{\plotpoint}}
\put(171.0,131.0){\rule[-0.200pt]{0.400pt}{175.375pt}}
\put(171.0,131.0){\rule[-0.200pt]{307.870pt}{0.400pt}}
\put(1449.0,131.0){\rule[-0.200pt]{0.400pt}{175.375pt}}
\put(171.0,859.0){\rule[-0.200pt]{307.870pt}{0.400pt}}
\end{picture}

%% file: CompLiap.bbl
\begin{thebibliography}{35}
\expandafter\ifx\csname natexlab\endcsname\relax\def\natexlab#1{#1}\fi
\providecommand{\bibinfo}[2]{#2}
\ifx\xfnm\relax \def\xfnm[#1]{\unskip,\space#1}\fi
\bibitem[{Alefeld and Herzberger(1983)}]{ah83}
\bibinfo{author}{G.~Alefeld}, \bibinfo{author}{J.~Herzberger},
  \bibinfo{title}{Introduction to Interval Computations},
  \bibinfo{publisher}{Academic Press}, \bibinfo{address}{New York},
  \bibinfo{year}{1983}.
\bibitem[{Blanck(2005)}]{Bla05}
\bibinfo{author}{J.~Blanck}, \bibinfo{title}{Efficient exact computation of
  iterated maps}, \bibinfo{journal}{The Journal of Logic and Algebraic
  Programming} \bibinfo{volume}{64} (\bibinfo{year}{2005})
  \bibinfo{pages}{41--59}.
\bibitem[{Blanck(2006)}]{bl06}
\bibinfo{author}{J.~Blanck}, \bibinfo{title}{Exact real arithmetic using
  centred intervals and bounded error terms}, \bibinfo{journal}{The Journal of
  Logic and Algebraic Programming} \bibinfo{volume}{66} (\bibinfo{year}{2006})
  \bibinfo{pages}{50--67}.
\bibitem[{Brattka and Hertling(1998)}]{BH98}
\bibinfo{author}{V.~Brattka}, \bibinfo{author}{P.~Hertling},
  \bibinfo{title}{Feasible real random access machines},
  \bibinfo{journal}{Journal of Complexity} \bibinfo{volume}{14}
  (\bibinfo{year}{1998}) \bibinfo{pages}{490--526}.
\bibitem[{Collet and Eckmann(1980)}]{ce80}
\bibinfo{author}{P.~Collet}, \bibinfo{author}{J.P. Eckmann},
  \bibinfo{title}{Iterated Maps on the Interval as Dynamical Systems}, Progress
  in Physics, \bibinfo{publisher}{Birkh\"{a}user}, \bibinfo{address}{Boston,
  Massachusetts}, \bibinfo{year}{1980}.
\bibitem[{Collet and Eckmann(2006)}]{ce06}
\bibinfo{author}{P.~Collet}, \bibinfo{author}{J.P. Eckmann},
  \bibinfo{title}{Concepts and Results in Chaotic Dynamics}, Theoretical and
  Mathematical Physics, \bibinfo{publisher}{Springer-Verlag},
  \bibinfo{address}{Berlin, Heidelberg}, \bibinfo{year}{2006}.
\bibitem[{Devaney(1989)}]{de89}
\bibinfo{author}{R.L. Devaney}, \bibinfo{title}{An Introduction to Chaotic
  Dynamical Systems}, \bibinfo{publisher}{Addison-Wesley},
  \bibinfo{address}{Redwood City, California}, \bibinfo{edition}{2nd} edition,
  \bibinfo{year}{1989}.
\bibitem[{Fousse et~al.(2007)Fousse, Hanrot, Lef\`evre, P\'elissier and
  Zimmermann}]{Fousse:2007:MMP}
\bibinfo{author}{L.~Fousse}, \bibinfo{author}{G.~Hanrot},
  \bibinfo{author}{V.~Lef\`evre}, \bibinfo{author}{P.~P\'elissier},
  \bibinfo{author}{P.~Zimmermann}, \bibinfo{title}{{MPFR}: A multiple-precision
  binary floating-point library with correct rounding}, \bibinfo{journal}{{ACM}
  Transactions on Mathematical Software} \bibinfo{volume}{33}
  (\bibinfo{year}{2007}) \bibinfo{pages}{13:1--13:15}.
\bibitem[{Goldberg(1991)}]{go91}
\bibinfo{author}{D.~Goldberg}, \bibinfo{title}{What every computer scientist
  should know about floating-point arithmetic}, \bibinfo{journal}{ACM Computing
  Surveys} \bibinfo{volume}{23} (\bibinfo{year}{1991}) \bibinfo{pages}{5--48}.
\bibitem[{Grebogi et~al.(1990)Grebogi, Hammel, Yorke and Sauer}]{gh90}
\bibinfo{author}{C.~Grebogi}, \bibinfo{author}{S.M. Hammel},
  \bibinfo{author}{J.A. Yorke}, \bibinfo{author}{T.~Sauer},
  \bibinfo{title}{Shadowing of physical trajectories in chaotic dynamics:
  Containment and refinement}, \bibinfo{journal}{Physical Review Letters}
  \bibinfo{volume}{65} (\bibinfo{year}{1990}) \bibinfo{pages}{1527--1530}.
\bibitem[{Hammel et~al.(1987)Hammel, Yorke and Grebogi}]{hy87}
\bibinfo{author}{S.M. Hammel}, \bibinfo{author}{J.A. Yorke},
  \bibinfo{author}{C.~Grebogi}, \bibinfo{title}{Do numerical orbits of chaotic
  dynamical processes represent true orbits?}, \bibinfo{journal}{Journal of
  Complexity} \bibinfo{volume}{3} (\bibinfo{year}{1987})
  \bibinfo{pages}{136--145}.
\bibitem[{Higham(2002)}]{hi02}
\bibinfo{author}{N.J. Higham}, \bibinfo{title}{Accuracy and Stability of
  Numerical Algorithms}, \bibinfo{publisher}{SIAM},
  \bibinfo{address}{Philadelphia}, \bibinfo{edition}{2.} edition,
  \bibinfo{year}{2002}.
\bibitem[{Hirsch et~al.(2004)Hirsch, Smale and Devaney}]{hsd04}
\bibinfo{author}{M.W. Hirsch}, \bibinfo{author}{S.~Smale},
  \bibinfo{author}{R.L. Devaney}, \bibinfo{title}{Differential Equations,
  Dynamical Systems and an Introduction to Chaos}, \bibinfo{publisher}{Elsevier
  Academic Press}, \bibinfo{address}{Amsterdam}, \bibinfo{year}{2004}.
\bibitem[{{IEEE~2008}(2008)}]{ieee08}
\bibinfo{author}{{IEEE~2008}}, \bibinfo{title}{{IEEE} Standard for
  Floating-Point Arithmetic (ANSI/IEEE Std 754-2008)},
  \bibinfo{publisher}{IEEE}, \bibinfo{address}{New York}, \bibinfo{year}{2008}.
\bibitem[{Katok and Hasselblatt(1995)}]{kh95}
\bibinfo{author}{A.~Katok}, \bibinfo{author}{B.~Hasselblatt},
  \bibinfo{title}{Introduction to the Modern Theory of Dynamical Systems},
  \bibinfo{publisher}{Cambridge University Press}, \bibinfo{address}{Cambridge
  New York Melbourne}, \bibinfo{year}{1995}.
\bibitem[{Ko(1991)}]{ko91}
\bibinfo{author}{K.I. Ko}, \bibinfo{title}{Complexity Theory of Real
  Functions}, \bibinfo{publisher}{Birkh\"{a}user}, \bibinfo{address}{Boston
  Basel Berlin}, \bibinfo{year}{1991}.
\bibitem[{Lohner(1993)}]{lo93}
\bibinfo{author}{R.J. Lohner}, \bibinfo{title}{Interval arithmetic in staggered
  correction format}, in: \bibinfo{editor}{E.~Adams},
  \bibinfo{editor}{U.~Kulisch} (Eds.), \bibinfo{booktitle}{Scientific Computing
  with Automatic Result Verification}, volume \bibinfo{volume}{189} of
  \textit{\bibinfo{series}{Mathematics in Science and Engineering}},
  \bibinfo{publisher}{Academic Press}, \bibinfo{address}{San Diego},
  \bibinfo{year}{1993}.
\bibitem[{Metropolis et~al.(1973)Metropolis, Rota and Tanny}]{mrt75}
\bibinfo{author}{N.~Metropolis}, \bibinfo{author}{G.C. Rota},
  \bibinfo{author}{S.M. Tanny}, \bibinfo{title}{Significance arithmetic: The
  carrying algorithm}, \bibinfo{journal}{Journal of Combinatorial Theory A}
  \bibinfo{volume}{14} (\bibinfo{year}{1973}) \bibinfo{pages}{386--421}.
\bibitem[{Moore(1966)}]{mo66}
\bibinfo{author}{R.E. Moore}, \bibinfo{title}{Interval Analysis},
  \bibinfo{publisher}{Prentice-Hall Inc.}, \bibinfo{address}{Englewood-Cliffs
  N.J.}, \bibinfo{year}{1966}.
\bibitem[{Mrozek(1995)}]{mr95}
\bibinfo{author}{M.~Mrozek}, \bibinfo{title}{Rigorous numerics of chaotic
  dynamical systems}, in: \bibinfo{editor}{P.~Garbaczewski},
  \bibinfo{editor}{M.~Wolf}, \bibinfo{editor}{W.~Aleksander} (Eds.),
  \bibinfo{booktitle}{Chaos - The Interplay Between Stochastic and
  Deterministic Behaviour}, volume \bibinfo{volume}{457} of
  \textit{\bibinfo{series}{Lecture Notes in Physics}},
  \bibinfo{publisher}{Springer-Verlag}, \bibinfo{address}{Berlin Heidelberg New
  York}, \bibinfo{year}{1995}, pp. \bibinfo{pages}{283--296}.
\bibitem[{M{\"u}ller(2001)}]{Mue01}
\bibinfo{author}{N.T. M{\"u}ller}, \bibinfo{title}{The {iRRAM}: Exact
  arithmetic in {C}++}, in: \bibinfo{editor}{J.~Blanck},
  \bibinfo{editor}{V.~Brattka}, \bibinfo{editor}{P.~Hertling} (Eds.),
  \bibinfo{booktitle}{Computability and Complexity in Analysis}, volume
  \bibinfo{volume}{2064} of \textit{\bibinfo{series}{Lecture Notes in Computer
  Science}}, \bibinfo{publisher}{Springer}, \bibinfo{address}{Berlin},
  \bibinfo{year}{2001}, pp. \bibinfo{pages}{222--252}. \bibinfo{note}{4th
  International Workshop, CCA 2000, Swansea, UK, September 2000}.
\bibitem[{M{\"u}ller(2005)}]{mu05}
\bibinfo{author}{N.T. M{\"u}ller}, \bibinfo{title}{Efficient implementation of
  exact real numbers}, \bibinfo{year}{2005}. \bibinfo{note}{Tutorial. An
  electronic copy of it is available at http://www.cc.kyoto-su.ac.jp/$\tilde{\
  }$yasugi/page/Kakenhi/mueller.pdf}.
\bibitem[{Neumaier and Rage(1993)}]{nr93}
\bibinfo{author}{A.~Neumaier}, \bibinfo{author}{T.~Rage},
  \bibinfo{title}{Rigorous chaos verification in discrete dynamical systems},
  \bibinfo{journal}{Physica D} \bibinfo{volume}{67} (\bibinfo{year}{1993})
  \bibinfo{pages}{327--346}.
\bibitem[{Pour-El and Richards(1989)}]{pr89}
\bibinfo{author}{M.B. Pour-El}, \bibinfo{author}{J.I. Richards},
  \bibinfo{title}{Computability in Analysis and Physics},
  \bibinfo{publisher}{Springer-Verlag}, \bibinfo{address}{Berlin Heidelberg New
  York}, \bibinfo{year}{1989}.
\bibitem[{Rage et~al.(1994)Rage, Neumaier and Schlier}]{rn94}
\bibinfo{author}{T.~Rage}, \bibinfo{author}{A.~Neumaier},
  \bibinfo{author}{C.~Schlier}, \bibinfo{title}{Rigorous verification of chaos
  in a molecular model}, \bibinfo{journal}{Physical Review E}
  \bibinfo{volume}{50} (\bibinfo{year}{1994}) \bibinfo{pages}{2682--2688}.
\bibitem[{Ratschek and Rokne(1984)}]{rr84}
\bibinfo{author}{H.~Ratschek}, \bibinfo{author}{J.~Rokne},
  \bibinfo{title}{Computer Methods for the Range of Functions},
  \bibinfo{publisher}{Ellis Horwood Limited}, \bibinfo{address}{Chichester},
  \bibinfo{year}{1984}.
\bibitem[{Ratschek and Rokne(1988)}]{rr88}
\bibinfo{author}{H.~Ratschek}, \bibinfo{author}{J.~Rokne}, \bibinfo{title}{New
  Computer Methods for Global Optimization}, \bibinfo{publisher}{Ellis Horwood
  Limited}, \bibinfo{address}{Chichester}, \bibinfo{year}{1988}.
\bibitem[{Revol and Rouillier(2005)}]{ReRo02}
\bibinfo{author}{N.~Revol}, \bibinfo{author}{F.~Rouillier},
  \bibinfo{title}{{Motivations for an Arbitrary Precision Interval Arithmetic
  and the MPFI Library}}, \bibinfo{journal}{Reliable Computing}
  \bibinfo{volume}{11} (\bibinfo{year}{2005}) \bibinfo{pages}{275--290}.
\bibitem[{Rump(1999)}]{ru99}
\bibinfo{author}{S.M. Rump}, \bibinfo{title}{Fast and parallel interval
  arithmetic}, \bibinfo{journal}{BIT} \bibinfo{volume}{39}
  (\bibinfo{year}{1999}) \bibinfo{pages}{534--554}.
\bibitem[{Sauer and Yorke(1991)}]{sy91}
\bibinfo{author}{T.~Sauer}, \bibinfo{author}{J.A. Yorke},
  \bibinfo{title}{Rigorous verification of trajectories for the computer
  simulation of dynamical systems}, \bibinfo{journal}{Nonlinearity}
  \bibinfo{volume}{4} (\bibinfo{year}{1991}) \bibinfo{pages}{961--979}.
\bibitem[{Sofroniou and Spaletta(2005)}]{ss05}
\bibinfo{author}{M.~Sofroniou}, \bibinfo{author}{G.~Spaletta},
  \bibinfo{title}{Precise numerical computation}, \bibinfo{journal}{The Journal
  of Logic and Algebraic Programming} \bibinfo{volume}{64}
  (\bibinfo{year}{2005}) \bibinfo{pages}{113--134}.
\bibitem[{Sterbenz(1974)}]{st74}
\bibinfo{author}{P.H. Sterbenz}, \bibinfo{title}{Floating Point Computation},
  \bibinfo{publisher}{Prentice Hall}, \bibinfo{address}{Englewood Cliffs, NJ},
  \bibinfo{year}{1974}.
\bibitem[{Weihrauch(1987)}]{Wei87}
\bibinfo{author}{K.~Weihrauch}, \bibinfo{title}{Computability},
  volume~\bibinfo{volume}{9} of \textit{\bibinfo{series}{EATCS Monographs on
  Theoretical Computer Science}}, \bibinfo{publisher}{Springer},
  \bibinfo{address}{Berlin}, \bibinfo{year}{1987}.
\bibitem[{Weihrauch(2000)}]{wh00}
\bibinfo{author}{K.~Weihrauch}, \bibinfo{title}{Computable Analysis},
  \bibinfo{publisher}{Springer-Verlag}, \bibinfo{address}{Berlin Heidelberg New
  York}, \bibinfo{year}{2000}.
\bibitem[{Wilkinson(1963)}]{wi63}
\bibinfo{author}{J.H. Wilkinson}, \bibinfo{title}{Rounding Errors in Algebraic
  Processes}, \bibinfo{publisher}{Prentice-Hall}, \bibinfo{address}{Englewood
  Cliffs, N.J.}, \bibinfo{year}{1963}.

\end{thebibliography}
